\documentclass[11pt]{amsart}
\usepackage[pagewise]{lineno}
\usepackage{amsmath,amssymb,latexsym,soul,cite,mathrsfs}

\usepackage{color,enumitem,graphicx}
\usepackage[colorlinks=true,urlcolor=blue,
citecolor=red,linkcolor=red,linktocpage,pdfpagelabels,
bookmarksnumbered,bookmarksopen]{hyperref}
\usepackage[english]{babel}

\usepackage[left=2.9cm,right=2.9cm,top=2.8cm,bottom=2.8cm]{geometry}
\usepackage[hyperpageref]{backref}

\usepackage[colorinlistoftodos]{todonotes}
\makeatletter
\providecommand\@dotsep{5}
\def\listtodoname{List of Todos}
\def\listoftodos{\@starttoc{tdo}\listtodoname}
\makeatother

\numberwithin{equation}{section}

\newcommand{\Om} {\Omega}

\newtheorem{Theorem}{Theorem}[section]
\newtheorem{Lemma}[Theorem]{Lemma}

\newtheorem{Corollary}[Theorem]{Corollary}
\newtheorem{Remark}[Theorem]{Remark}
\newtheorem{Definition}[Theorem]{Definition}

\begin{document}

\title[Weighted anisotropic Sobolev inequality...]
{Weighted anisotropic Sobolev inequality with extremal and associated singular problems}
\author{Kaushik Bal and Prashanta Garain}

\address[Kaushik Bal ]
{\newline\indent Department of Mathematics and Statistics
\newline\indent
Indian Institute of Technology Kanpur
\newline\indent
Kanpur-208016, Uttar Pradesh, India
\newline\indent
Email: {\tt kaushik@iitk.ac.in} }

\address[Prashanta Garain ]
{\newline\indent Department of Mathematics
	\newline\indent
Uppsala University
	\newline\indent
S-751 06 Uppsala, Sweden
\newline\indent
Email: {\tt pgarain92@gmail.com} }

\pretolerance10000

\begin{abstract}
For a given Finsler-Minkowski norm $\mathcal{F}$ in $\mathbb{R}^N$ and a bounded smooth domain $\Omega\subset\mathbb{R}^N$ $\big(N\geq 2\big)$, we establish the following weighted anisotropic Sobolev inequality  
$$
S\left(\int_{\Omega}|u|^q f\,dx\right)^\frac{1}{q}\leq\left(\int_{\Omega}\mathcal{F}(\nabla u)^p w\,dx\right)^\frac{1}{p},\quad\forall\,u\in W_0^{1,p}(\Omega,w)\leqno{\mathcal{(P)}}
$$
where $W_0^{1,p}(\Omega,w)$ is the weighted Sobolev space under a class of $p$-admissible weights $w$, where $f$ is some nonnegative integrable function in $\Omega$. We discuss the case $0<q<1$ and observe that
$$
\mu(\Omega):=\inf_{u\in W_{0}^{1,p}(\Omega,w)}\Bigg\{\int_{\Omega}\mathcal{F}(\nabla u)^p w\,dx:\int_{\Omega}|u|^{q}f\,dx=1\Bigg\}\leqno{\mathcal{(Q)}}
$$
is associated with singular weighted anisotropic $p$-Laplace equations. To this end, we also study existence and regularity properties of solutions for weighted anisotropic $p$-Laplace equations under the mixed and exponential singularities.
\end{abstract}


\maketitle

\section{Introduction}
In this article, we establish the following weighted {anisotropic} Sobolev inequality,
$$
S\left(\int_{\Omega}|u|^q f\,dx\right)^\frac{1}{q}\leq\left(\int_{\Omega}\mathcal{F}(\nabla u)^p w\,dx\right)^\frac{1}{p},\quad\forall\,u\in W_0^{1,p}(\Omega,w),\leqno{\mathcal{(P)}}
$$
where $S$ is the Sobolev constant, $0<q<1<p<\infty$, $\Omega\subset\mathbb{R}^N$ $\big(N\geq 2\big)$ is a bounded smooth domain. Here $\mathcal{F}:\mathbb{R}^N\to[0,\infty)$ is a Finsler-Minkowski norm, i.e. $\mathcal{F}$ satisfies the hypothesis from $(H0)-(H4)$ given by
\begin{enumerate}
\item[(H0)] $\mathcal{F}(x)\geq 0$, for every $x\in\mathbb{R}^N$.
\item[(H1)] $\mathcal{F}(x)=0$, if and only if $x=0$.
\item[(H2)] $\mathcal{F}(tx)=|t|\mathcal{F}(x)$, for every $x\in\mathbb{R}^N$ and $t\in\mathbb{R}$.
\item[(H3)] $\mathcal{F}\in C^{\infty}\Big(\mathbb{R}^N\setminus\{0\}\Big)$.
\item[(H4)] the Hessian matrix $\nabla_{\eta}^2\Big(\frac{\mathcal{F}^2}{2}\Big)(x)$ is positive definite for all $x\in\mathbb{R}^N\setminus\{0\}$.
\end{enumerate} 

We assume the weight function $w$ in a class of $p$-admissible weights to be discussed in Section $2$ and $f\in L^m(\Omega)\setminus\{0\}$ is some nonnegative function. Our main emphasis is the case of $0<q<1$ in $\mathcal{(P)}$ and we observe that
 $$
\mu(\Omega):=\inf_{u\in W_{0}^{1,p}(\Omega,w)}\Bigg\{\int_{\Omega}\mathcal{F}(\nabla u)^p w\,dx:\int_{\Omega}|u|^{q}f\,dx=1\Bigg\},\leqno{\mathcal{(Q)}}
 $$
is associated with the following type of weighted singular anisotropic $p$-Laplace equation:
 $$
 -\mathcal{F}_{p,w}u=\frac{f(x)}{u^\delta}+\frac{g(x)}{u^\gamma}\text{ in }\Omega,\,u>0\text{ in }\Omega,\,u=0\text{ on }\partial\Omega\leqno{\mathcal{(S)}}
 $$
where $0<\delta,\gamma<1$, $(f,g)\neq(0,0)\in L^m(\Omega)$ are nonnegative functions. Here, 
\begin{equation}\label{Finsdef}
\mathcal{F}_{p,w}u:=\text{div}\left(w\mathcal{F}(\nabla u)^{p-1}\nabla_{\eta}\mathcal{F}(\nabla u)\right),
\end{equation}
is the weighted anisotropic $p$-Laplace operator, where $\nabla_{\eta}$ denotes the gradient with respect to $\eta$. Further, we prove existence and regularity results for the exponential singular problem,
$$
 -\mathcal{F}_{p,w}u=h(x)e^\frac{1}{u}\text{ in }\Omega,\,u>0\text{ in }\Omega,\,u=0\text{ on }\partial\Omega,\leqno{\mathcal{(R)}}
 $$
 where $h\in L^t(\Om)$ is nonnegative for some $t\geq 1$.\\
 Before proceeding further, let us discuss some examples of $\mathcal{F}$.\\
 \textbf{Examples:} Let $x=(x_1,x_2,\cdots,x_N)\in\mathbb{R}^N$. 
\begin{enumerate}
\item[(i)] Then for $t>1$, we define 
\begin{equation}\label{ex1}
\mathcal{F}_t(x):=\Big(\sum_{i=1}^{N}|x_i|^t\Big)^\frac{1}{t}.
\end{equation}
\item[(ii)] For $\lambda,\mu>0$, we define
\begin{equation}\label{ex2}
\mathcal{F}_{\lambda,\mu}(x):=\sqrt{\lambda\sqrt{\sum_{i=1}^{N}x_i^{4}}+\mu\sum_{i=1}^{N}x_i^{2}}.
\end{equation}
\end{enumerate}
The functions $\mathcal{F}_t, \mathcal{F}_{\lambda,\mu}:\mathbb{R}^N\to[0,\infty)$ given by \eqref{ex1} and \eqref{ex2} satisfies all the hypothesis from $(H0)-(H4)$, see Mezei-Vas \cite{MV}.

\begin{Remark}\label{exrmk1}
For $i=1,2$, if $\lambda_i,\mu_i$ are positive real numbers such that $\frac{\lambda_1}{\mu_1}\neq\frac{\lambda_2}{\mu_2}$, then $\mathcal{F}_{\lambda_1,\mu_1}$ and $\mathcal{F}_{\lambda_2,\mu_2}$ given by \eqref{ex2} defines two non-isometric norms in $\mathbb{R}^N$.
\end{Remark}

\begin{Remark}\label{hypormk2}
Since all norms in $\mathbb{R}^N$ are equivalent, there exist positive constants $C_1, C_2$ such that 
$$
C_1|x|\leq \mathcal{F}(x)\leq C_2|x|,\quad\forall \,x\in\mathbb{R}^N.
$$
\end{Remark}

The dual $\mathcal{F}_0:\mathbb{R}^N\to[0,\infty)$ of $\mathcal{F}$ is defined by 
\begin{equation}\label{dual}
\mathcal{F}_0(\xi):=\sup_{x\in\mathbb{R}^N\setminus\{0\}}\frac{\langle x,\xi\rangle}{\mathcal{F}(x)}.
\end{equation}
We refer to Bao-Chern-Shen \cite{BCS}, Xia \cite{Xiathesis} and Rockafellar \cite{Rbook}  for more details on $\mathcal{F}_0$.

It is easy to observe that, if  $\mathcal{F}=\mathcal{F}_{t}$ given by \eqref{ex1}, then
\begin{equation}\label{exlap}
\mathcal{F}_{p,w}u:=
\begin{cases}
\Delta_{p,w}u=\text{div}(w|\nabla u|^{p-2}\nabla u), ( \text{weighted }  $p$-Laplacian )\text{ if }t=2,\,1<p<\infty,\\
\mathcal{S}_{p,w}u=\text{div}\left(w\left|\frac{\partial u}{\partial x_i}\right|^{p-2}\nabla u\right), ( \text{weighted pseudo } $p$-Laplacian )\text{ if }t=p\in(1,\infty).
\end{cases}
\end{equation}

Therefore, $\mathcal{F}_{p,w}$ extends the weighted $p$-Laplace and weighted pseudo $p$-Laplace operators and thus a large class of weighted quasilinear equations is covered by $\mathcal{F}_{p,w}$.

Let us discuss some known results related to our present study. For $q>1$, Sobolev inequalities of the form $(\mathcal{P})$ are widely studied throughout the last three decade and there is a colossal amount of literaure available in this direction. We refer to Aubin \cite{Aubin}, Talenti \cite{Talenti}, \^{O}tani \cite{Otani}, Franzina-Lamberti \cite{Franzina}, Belloni-Kawohl \cite{BKmm}, Lindqvist \cite{PLin} and the references therein. For the class of Muckenhoupt weights \cite{Muc}, weighted Sobolev inequalities are established in Fabes-Kenig-Serapioni \cite{EFabes}, for a more general class of $p$-admissible weights, refer to Heinonen-Kilpel\"{a}ine-Martio \cite{Juh} and the references therein.

Although very less is known in the anisotropic case. In this context, recently Ciraolo-Figalli-Roncoroni \cite{CFR} proved a sharp version of $\mathcal{(P)}$, for a certain class of weight functions $w=f$ in a convex cone $\Omega$, with $1<p<N$, where the exponent $q>1$, depends on the weight function. Indeed, for $w=1$, the exponent $q$ is the critical Sobolev exponent in \cite{CFR}. We also refer to Belloni-Ferone-Kawohl \cite{BFK}, di Blasio-Pisante-Giovanni\cite{a2}, El Hamidi-Rakotoson \cite{a3}, Filippas-Moschini-Tertikas \cite{a1}, Dipierro-Poggesi-Valdinoci \cite{DPV21} and the references therein for related works.

When $0<q<1$, for any $1<p<\infty$, Anello-Faraci-Iannizzotto \cite{GFA} shown that
\begin{equation}\label{sinsb}
\nu(\Omega):=\inf_{0\neq u\in W_{0}^{1,p}(\Omega)}\Bigg\{\int_{\Omega}|\nabla u|^p\,dx:\int_{\Omega}|u|^{q}\,dx=1\Bigg\},
\end{equation}
is achieved at a solution $u_q\in W_{0}^{1,p}(\Omega)$ of the following singular $p$-Laplace equation
\begin{equation}\label{AFIeqn}
-\Delta_p u=\nu(\Omega)u^{q-1}\text{ in }\Omega,\,u>0\text{ in }\Omega,\,u=0\text{ on }\partial\Omega.
\end{equation}
Such results has further been extended to a class of Muckenhoupt weights by Bal-Garain \cite{BGaniso}, in the nonlocal case by Ercole-Pereira \cite{EP1}, subelliptic setting by Garain-Ukhlov \cite{GU}, see also Ercole-Pereira \cite{EP} for related results. When $w$ is a $p$-admissible weight, Hara \cite{Hara} studied weighted Sobolev inequalities of type $(\mathcal{P})$ in terms of the associated singular problem.

In this article, we provide sufficient conditions on the weight function $w$ (which may vanish or blow up near the origin, for example $w(x)=|x|^{\alpha}$, $\alpha\in\mathbb{R}$) and $f$ that guarantees the weighted Sobolev inequality $(\mathcal{P})$ and also provide the extremal function of the associated variational problem.

We observe that such extremals correspond to the associated singular weighted anisotropic $p$-Laplace equation. We found a class of $p$-admissible weights useful for our purpose, for which weighted singular problems are recently studied in Garain-Kinnunen \cite{Garainpadm, GKin}. We would like to mention that in the classical case without weights, singular $p$-Laplace equations is studied widely till date, refer to \cite{CRT, GRbook, Canino, GST, BG, arcoya, Merida, Garainaniso, DeCave, G, GM, Santos, Mirifixed, Mirivar, Oljama, Olsaim, Oljde, OrPet, BGejde} and the references therein. In contrast to these literature, singular anisotropic $p$-Laplace equations are very less understood. In this context, we refer to Biset-Mebrate-Mohammed \cite{BMM20}, Farkas-Winkert \cite{PF20}, Farkas-Fiscella-Winkert \cite{PF21}, Bal-Garain-Mukherjee \cite{BGM} and the references therein. Here we are also able to extend some previous results known in the unweighted case for $p\geq 2$ to the weighted case case with $1<p<2$ (see Remark \ref{thm1newrmk2} and \ref{thm2newrmk}).

The idea here stem from the work of Anello-Faraci-Iannizzotto \cite{GFA} that is based on the approximation approach, where it was important to know the existence of the associated singular $p$-Laplace equation \eqref{AFIeqn}. We obtain existence and regularity results in the weighted anisotropic setting, for the more general mixed singular problem $\mathcal{(S)}$ (Theorem \ref{thm1}-\ref{thm1new}) and also study the exponential singular problem $\mathcal{(R)}$ (Theorem \ref{ethm}-\ref{ethmnew}). To this end, we follow the approach from Boccardo-Orsina \cite{BocOr}, where to deal with the mixed problem $\mathcal{(S)}$, we need to estimate both the singularities $u^{-\delta}$ and $u^{-\gamma}$ simultaneously as in \cite{BGaniso}. On the otherhand, for the case of $e^\frac{1}{u}$ in the problem $\mathcal{(R)}$, the exponent in the singularity is arbitrarily large and thus in general solutions lie outside $W_0^{1,p}(\Om,w)$. Due to this reason, one has to describe the boundary condition appropriately as has been illustrated in \cite{PeSil, BocOr, LzMc}. Here, we tackle our situation by following the domain approximation technique from Perera-Silva \cite{PeSil}.

\subsection*{Organization of the paper:} In Section $2$, we discuss some preliminaries in our setting and state the main results. In Section $3$, we obtain several auxiliary results, that are crucial to study the mixed singular problem $\mathcal{(S)}$ and the weighted anisotropic Sobolev inequality $\mathcal{(P)}$. In Section $4$, we prove some auxiliary results related to the exponential singular problem $\mathcal{(R)}$ and finally, in Section $4$, we prove our main results.

\section{Preliminaries and main results}
Throughout the rest of the article, we assume $1<p<\infty$, unless otherwise mentioned. We say that a function $w$ belong to the class of $p$-admissible weights $W_p$, if $w\in L^1_{\mathrm{loc}}(\mathbb{R}^N)$ such that $0<w<\infty$ almost everywhere in $\mathbb{R}^N$ and satisfies the following conditions:
\begin{enumerate}
\item[(i)] for any ball $B$ in $\mathbb{R}^N$, there exists a positive constant $C_{\mu}$ such that 
$$
\mu(2B)\leq C_{\mu}\,\mu(B),
$$
where 
$$
\mu(E)=\int_{E}w\,dx
$$
for a measurable subset $E$ in $\mathbb{R}^N$ and $d\mu(x)=w(x)\,dx$, where $dx$ is the $N$-dimensional Lebesgue measure.
\item[(ii)] If $D$ is an open set and $\{\phi_i\}_{i\in\mathbb{N}}\subset C^{\infty}(D)$ is a sequence of functions such that 
$$
\int_{D}|\phi_i|^p\,d\mu\to 0\text{ and } \int_{D}|\nabla\phi_i-v|^p\,d\mu\to 0
$$
as $i\to\infty$, where $v$ is a vector valued measurable function in $L^p(D,w)$, then $v=0$.
\item[(iii)] There exist constants $\kappa>1$ and $C_1>0$ such that 
\begin{equation}\label{wp}
\left(\frac{1}{\mu(B)}\int_{B}|\phi|^{\kappa p}\,d\mu\right)^\frac{1}{\kappa p}\leq C_1 r \left(\frac{1}{\mu(B)}\int_{B}|\nabla\phi|^p\,d\mu\right)^\frac{1}{p},
\end{equation}
whenever $B=B(x_0,r)$ is a ball in $\mathbb{R}^N$ centered at $x_0$ with radius $r$ and $\phi\in C_{c}^\infty(B)$.
\item[(iv)] There exists a constant $C_2>0$ such that
\begin{equation}\label{wp1}
\int_{B}|\phi-\phi_B|^p\,d\mu\leq C_2 r^p\int_{B}|\nabla\phi|^p\,d\mu,
\end{equation}
whenever $B=B(x_0,r)$ is a ball in $\mathbb{R}^N$ and $\phi\in C^\infty(B)$ is bounded. Here
$$
\phi_B=\frac{1}{\mu(B)}\int_{B}\phi\,d\mu.
$$
\end{enumerate}
The conditions (i)-(iv) are important in the theory of weighted Sobolev spaces, one can refer to \cite{Juh} for more details. \\
\textbf{Examples:} Muckenhoupt weights $A_p$ are $p$-admissible, see \cite[Theorem 15.21]{Juh}. In particular, if $c\leq w\leq d$ for some positive constants $c,d$, then $w\in A_p$ for any $1<p<\infty$. Let $1<p<N$ and $J_f(x)$ denote the determinant of the Jacobian matrix of a $K$-quasiconformal mapping $f:\mathbb{R}^N\to\mathbb{R}^N$, then $w(x)=J_f(x)^{1-\frac{p}{N}}\in W_q$ for any $q\geq p$, see \cite[Corollary 15.34]{Juh}. If $1<p<\infty$ and $\nu>-N$, then $w(x)=|x|^{\nu}\in W_p$, see \cite[Corollary 15.35]{Juh}. For more examples, refer to \cite{ex1, ex2, ex3, Tero, Juh} and the references therein. 

\begin{Definition}(Weighted Spaces)
Let $1<p<\infty$ and $w\in W_p$. Then the weighted Lebesgue space $L^{p}(\Omega,w)$ is the class of measurable functions $u:\Om\to\mathbb{R}$ such that the norm of $u$ given by
\begin{equation}\label{lnorm}
\|u\|_{L^p(\Om,w)} = \Big(\int_{\Omega}|u(x)|^{p} w(x)\,dx\Big)^\frac{1}{p}<\infty.
\end{equation}
The weighted Sobolev space $W^{1,p}(\Om,w)$ is the class of measurable functions $u:\Om\to\mathbb{R}$ such that
\begin{equation}\label{norm1}
\|u\|_{1,p,w} = \Big(\int_{\Omega}|u(x)|^{p} w(x)\,dx+\int_{\Omega}|\nabla u(x)|^{p} w(x)\,dx\Big)^\frac{1}{p}<\infty.
\end{equation}
If $u\in W^{1,p}(\Om',w)$ for every $\Om'\Subset\Om$, then we say that $u\in W^{1,p}_{\mathrm{loc}}(\Om,w)$. The weighted Sobolev space with zero boundary value is defined as
$$
W^{1,p}_{0}(\Omega,w)=\overline{\big(C_{c}^{\infty}(\Omega),\|\cdot\|_{1,p,w}\big)}.
$$
\end{Definition}
Using the Poincar\'e inequality from \cite{Juh}, the norm defined by \eqref{norm1} on the space $W_{0}^{1,p}(\Omega,w)$ is equivalent to the norm given by
\begin{equation}\label{equinorm}
\|u\|_{W_0^{1,p}(\Om,w)}=\Big(\int_{\Omega}|\nabla u|^p w\,dx\Big)^\frac{1}{p}.
\end{equation}
Moreover, the space $W^{1,p}_{0}(\Omega,w)$ is a separable and uniformly convex Banach space, see \cite{Juh}.

 Next, we state an embedding result which is crucial for us. To this end, for $1<p<\infty$, we define the set
\begin{equation}\label{I}
I=\Big[\frac{1}{p-1},\infty\Big)\cap\Big(\frac{N}{p},\infty\Big).
\end{equation}
Consider the following subclass of $W_p$ given by
\begin{equation}\label{wgtcls}
W_p^{s} = \Big\{w\in W_p: w^{-s}\in L^{1}(\Omega)\,\,\text{for some}\,\,s\in I\Big\}.
\end{equation}
Then for $s\in I$, the weight
$$
w(x)=|x|^{\nu}\in W_p^{s}\text{ for any }\nu\in\Big(-N,\frac{N}{s}\Big).
$$
Following the lines of the proof of \cite[Theorem $2.6$]{G} based on \cite{Drabek} the following embedding result holds.
\begin{Lemma}\label{embedding}
Let $1<p<\infty$ and $w \in W_p^{s}$ for some $s\in I$. Then the following continuous inclusion maps hold
\[
    W^{1,p}(\Omega,w)\hookrightarrow W^{1,p_s}(\Omega)\hookrightarrow 
\begin{cases}
    L^t(\Omega),& \text{for } p_s\leq t\leq p_s^{*}, \text{ if } 1\leq p_s<N, \\
    L^t(\Omega),& \text{ for } 1\leq t< \infty, \text{ if } p_s=N, \\
    C(\overline{\Omega}),& \text{ if } p_s>N,
\end{cases}
\]
where $p_s = \frac{ps}{s+1} \in [1,p)$.
Moreover, the second embedding above is compact except for $t=p_s^{*}=\frac{Np_s}{N-p_s}$, if $1\leq p_s<N$. Further, the same result holds for the space $W_{0}^{1,p}(\Omega,w)$.
\end{Lemma}



\begin{Remark}\label{Embrmk}
We note that if $0<c\leq w\leq d$ for some constants $c,d$, then $W^{1,p}(\Omega,w)=W^{1,p}(\Omega)$ and thus by the Sobolev embedding, Lemma \ref{embedding} holds by replacing $p_s$ and $p_s^{*}$ with $p$ and $p^{*}=\frac{Np}{N-p}$ respectively.
\end{Remark}

The following result follows from Farkas-Winkert \cite[Proposition $2.1$]{PF20} and Xia \cite[Proposition $1.2$]{Xiathesis}.
\begin{Lemma}\label{Happ}
For every $x\in\mathbb{R}^N\setminus\{0\}$ and $t\in\mathbb{R}\setminus\{0\}$, we have
\begin{enumerate}
\item[(A)] $x\cdot\nabla_{\eta} \mathcal{F}(x)=\mathcal{F}(x)$.
\item[(B)] $\nabla_{\eta}\mathcal{F}(tx)=\text{sign}(t)\nabla_{\eta}\mathcal{F}(x)$.
\item[(C)] $|\nabla_{\eta}\mathcal{F}(x)|\leq C$, for some positive constant $C$.
\item[(D)] $\mathcal{F}$ is strictly convex.
\end{enumerate}
\end{Lemma}

Next, the classical algebraic inequality holds from Peral \cite[Lemma A.0.5]{Pe}.
\begin{Lemma}\label{AI}
For any $a,b\in\mathbb{R}^N$, there exists a constant $C=C(p)>0$, such that
\begin{equation}\label{ALGin}
\langle |a|^{p-2}a-|b|^{p-2}b, a-b \rangle\geq
\begin{cases}
C|a-b|^p,\text{ if }2\leq p<\infty,\\
C\frac{|a-b|^2}{(|a|+|b|)^{2-p}},\text{ if }1<p<2.
\end{cases}
\end{equation}
\end{Lemma}

More generally, we state the Finsler algebraic inequality from \cite[Lemma $2.5$]{BGM}.
\begin{Lemma}\label{alg}
Let $2\leq p<\infty$. Then, for every $x,y\in\mathbb{R}^N$, there exists a constant $C=C(p)>0$, such that
\begin{equation}\label{algineq}
\begin{split}
\langle \mathcal{F}(x)^{p-1}\nabla_{\eta}\mathcal{F}(x)-\mathcal{F}(y)^{p-1}\nabla_{\eta}\mathcal{F}(y),x-y\rangle&\geq
C\mathcal{F}(x-y)^p.
\end{split}
\end{equation}
\end{Lemma}
\begin{Remark}\label{algrmk}
When $2\leq p<\infty$ and $\mathcal{F}(x)=\mathcal{F}_2(x)=|x|$ as given by \eqref{ex1}, Lemma \ref{alg} coincides with Lemma \ref{AI}.
\end{Remark}

\begin{Corollary}\label{regapp}
From Lemma \ref{Happ} and Remark \ref{hypormk2} we have
\begin{equation}\label{lbd}
\mathcal{F}(x)^{p-1}\nabla_{\eta}\mathcal{F}(x)\cdot x=\mathcal{F}(x)^p\geq C_1|x|^p,\quad\forall\,x\in\mathbb{R}^N,
\end{equation}
\begin{equation}\label{ubd}
\big|\mathcal{F}(x)^{p-1}\nabla_{\eta}\mathcal{F}(x)\big|\leq C_2|x|^{p-1},\quad\forall\,x\in\mathbb{R}^N\text{ and}
\end{equation}
\begin{equation}\label{homo}
\mathcal{F}(tx)^{p-1}\nabla_{\eta}\mathcal{F}(tx)=|t|^{p-2}t \mathcal{F}(x)^{p-1}\nabla_{\eta}\mathcal{F}(x),\quad\forall\,x\in\mathbb{R}^N\text{ and }t\in\mathbb{R}\setminus\{0\}.
\end{equation}
Moreover, by Lemma \ref{alg}, we have
\begin{equation}\label{regapp1}
\langle \mathcal{F}(x)^{p-1}\nabla_{\eta}\mathcal{F}(x)-\mathcal{F}(y)^{p-1}\nabla_{\eta}\mathcal{F}(y),x-y\rangle>0,\quad\forall\,x\neq y\in\mathbb{R}^N.
\end{equation}
\end{Corollary}
\subsection*{Notation:} Throughout the rest of the article, we shall use the following notations.
\begin{itemize}
\item For $u\in W_0^{1,p}(\Om,w)$, denote by $\|u\|$ to mean the norm $\|u\|_{W_0^{1,p}(\Om,w)}$ as defined by \eqref{equinorm}.
\item For given constants $c,d$ and a set $S$, by $c\leq u\leq d$ in $S$, we mean $c\leq u\leq d$ almost everywhere in $S$. Moreover, we write $|S|$ to denote the Lebesgue measure of $S$.
\item $\langle,\rangle$ denotes the standard inner product in $\mathbb{R}^N$.
\item The conjugate exponent of $\theta>1$ by $\theta':=\frac{\theta}{\theta-1}$.
\item For $1<p<N$, we denote by $p^*:=\frac{Np}{N-p}$ to mean the critical Sobolev exponent.
\item For $a\in\mathbb{R}$, we denote by $a^+:=\max\{a,0\}$, $a^-=\max\{-a,0\}$ and $a_-:=\min\{a,0\}$.
\item We write by $c,C$ or $C_i$ for $i\in\mathbb{N}$ to mean a constant which may vary from line to line or even in the same line. If a constant $C$ depends on $r_1,r_2,\ldots$, we denote it by $C(r_1,r_2,\ldots)$. 
\end{itemize}

Now, we define the notion of weak solutions for the problem $\mathcal{(S)}$ as follows:
\begin{Definition}(Weak solution for $\mathcal{(S)}$)
Let $1<p<\infty$, $0<\delta,\gamma<1$ and $w\in W_p$. Then, we say that $u\in W_0^{1,p}(\Om,w)$ is a weak solution of the problem $\mathcal{(S)}$, if $u>0$ in $\Omega$ and for every $\omega\Subset\Omega$, there exists a constant $c_\omega$ such that $u\geq c_{\omega}>0$ in $\omega$ and for every $\phi\in C_c^{1}(\Omega)$, we have 
\begin{equation}\label{weakform1}
<-\mathcal{F}_{p,w} u,\phi>:=\int_{\Omega}w(x)\mathcal{F}(\nabla u)^{p-1}\mathcal{F}(\nabla u)\nabla\phi\,dx=\int_{\Omega}\Big(\frac{f}{u^{\delta}}+\frac{g}{u^{\gamma}}\Big)\phi\,dx.
\end{equation}
\end{Definition}

Further, we define weak solutions for the problem $\mathcal{(R)}$ as follows.
\begin{Definition}(Weak solution for $\mathcal{(R)}$)
Let $1<p<\infty$ and $w\in W_p$. Then, we say that $u\in W^{1,p}_{\mathrm{loc}}(\Om,w)$ is a weak solution of the problem $\mathcal{(R)}$, if $u>0$ in $\Omega$ such that $(u-\epsilon)^+\in W_0^{1,p}(\Om,w)$ for every $\epsilon>0$ and for every $\omega\Subset\Omega$, there exists a constant $c_\omega$ such that $u\geq c_{\omega}>0$ in $\omega$ and for every $\phi\in C_c^{1}(\Omega)$, we have 
\begin{equation}\label{weakform2}
<-\mathcal{F}_{p,w} u,\phi>:=\int_{\Omega}w(x)\mathcal{F}(\nabla u)^{p-1}\mathcal{F}(\nabla u)\nabla\phi\,dx=\int_{\Omega}h(x)e^\frac{1}{u}\phi\,dx.
\end{equation}
\end{Definition}

\subsection*{Statement of the main results} 
Now we state our main results as follows:
\begin{Theorem}\label{thm1}
Let $0<\delta,\gamma<1$, $2\leq p<\infty$ and $w\in W_p^{s}$ for some $s\in I$. Assume that $(f,g)\neq (0,0)$ is nonnegative. 
\begin{enumerate}
\item[(a)] Let $(f,g)\in L^1(\Omega)\times L^1(\Omega)$. Then the problem $\mathcal{(S)}$ admits at most one weak solution in $W_{0}^{1,p}(\Omega,w)$.
\item[(b)] Let $(f,g)\in L^{m_{\delta}}(\Omega)\times L^{r_{\gamma}}(\Omega)$, where
\[
    m_{\delta}:=
\begin{cases}
    \big(\frac{p_s^{*}}{1-\delta}\big)',& \text{for } 1\leq p_s<N, \\
    m>1,& \text{for } p_s=N, \\
    1,& \text{ for } p_s>N
\end{cases}
\]
and
\[
    r_{\gamma}:=
\begin{cases}
    \big(\frac{p_s^{*}}{1-\gamma}\big)',& \text{for } 1\leq p_s<N, \\
    m>1,& \text{for } p_s=N, \\
    1,& \text{ for } p_s>N.
\end{cases}
\]
Then the problem $\mathcal{(S)}$ has a weak solution $u_{\delta,\gamma}$ in $W_0^{1,p}(\Om,w)$.
\end{enumerate}
\end{Theorem}


In addition, we prove the following regularity results for $u_{\delta,\gamma}$.
\begin{Theorem}\label{regthm1}
Let $2\leq p<\infty$ and $w\in W_p^{s}$ for some $s\in I$. Suppose that $(f,g)\neq(0,0)$ is nonnegative such that
\begin{enumerate}
\item[(a)] $(f,g)\in L^q(\Om)\times L^q(\Om)$ for $q>\frac{p_s^{*}}{p_s^{*}-p}$, if $1\leq p_s<N$,
\item[(b)] $(f,g)\in L^q(\Om)\times L^q(\Om)$ for $q>\frac{r}{r-p}$, if $p_s=N$, $r>p$,
\item[(c)] $(f,g)\in L^1(\Om)\times L^1(\Om)$ if $p_s>N$. 
\end{enumerate} 
Then, $u_{\delta,\gamma}\in L^\infty(\Om)$.
\end{Theorem}

\begin{Theorem}\label{thm1new}
If $w\in W_p^{s}$ for some $s\in I$ and $\mathcal{F}_{p,w}=\Delta_{p,w}$ or $\mathcal{S}_{p,w}$ as given by \eqref{exlap}, then Theorem \ref{thm1}-\ref{regthm1} holds, for any $1<p<\infty$.
\end{Theorem}

\begin{Remark}\label{thm1newrmk1}
For $\mathcal{F}_{p,w}=\Delta_{p,w}$ as given by \eqref{exlap}, Theorem \ref{thm1new} extends \cite[Theorem $2.7$]{BGaniso}. Further for $g=0$, Theorem \ref{thm1new} extends \cite[Theorem $5.2$]{BocOr} with $p=2$; \cite[Lemma $4.3$]{DeCave} for $1<p<N$; \cite[Theorem $4.6$]{Canino} for $1<p<\infty$ and extends \cite[Theorem $3.2$]{G} for $1<p<\infty$.
\end{Remark}

\begin{Remark}\label{thm1newrmk2}
If $\mathcal{F}_{p,w}=\mathcal{S}_{p,w}$ as given by \eqref{exlap}, then, for $g=0$, Theorem \ref{thm1new} extends \cite[Theorem $3.2$]{Mirivar} to the weighted case for $w\in W_p^{s}$, improves the range of $f$ and the restriction $p\geq 2$ to any $1<p<\infty$. Furthermore, Theorem \ref{thm1new} extends \cite[Theorem $2.9$]{BGaniso} to the weighted case for $w\in W_p^{s}$ and also improves the restriction $p\geq 2$ to any $1<p<\infty$.
\end{Remark}



\begin{Remark}\label{rmk}
If $g=0$, we denote the solutions $u_{\delta,\gamma}$ found in Theorem \ref{thm1} and \ref{thm1new} by $u_\delta$.
\end{Remark}
Now, we present our weighted anisotropic Sobolev inequality with extremal associated with $u_\delta$ as follows:
\begin{Theorem}\label{thm2}
Let $0<\delta<1$ and $2\leq p<\infty$. Assume that $w\in W_p^{s}$, for some $s\in I$ and $f\in L^{m_{\delta}}(\Omega)\setminus\{0\}$ be nonnegative, where $m_{\delta}$ is given by Theorem \ref{thm1}. 
Then,
\begin{enumerate} 
\item[(a)] 
\begin{align*}
\mu(\Omega)&:=\inf_{u\in W_0^{1,p}(\Om,w)}\Bigg\{\int_{\Omega}\mathcal{F}(\nabla u)^p w\,dx:\int_{\Omega}|u|^{1-\delta}f\,dx=1\Bigg\}\\
&=\left(\int_{\Omega}\mathcal{F}(\nabla u_\delta)^p w\,dx\right)^\frac{1-\delta-p}{1-\delta}.
\end{align*}
\item[(b)] Moreover, for every $v\in W_0^{1,p}(\Om,w)$, the following Sobolev type inequality
\begin{equation}\label{inequality1}
C\Big(\int_{\Omega}|v|^{1-\delta}f\,dx\Big)^\frac{p}{1-\delta}\leq\int_{\Omega}\mathcal{F}(\nabla v)^p w\,dx,
\end{equation}
holds, if and only if $$C\leq \mu(\Omega).$$
\end{enumerate}
\end{Theorem}

\begin{Theorem}\label{thm2new}
Let $w\in W_p^{s}$ for some $s\in I$. If $\mathcal{F}(x)=\mathcal{F}_{2}(x)$ or $\mathcal{F}(x)=\mathcal{F}_{p}(x)$, as given by \eqref{ex1}, then Theorem \ref{thm2} holds for any $1<p<\infty$.
\end{Theorem}

\begin{Remark}\label{thm2newrmk}
When $\mathcal{F}(x)=\mathcal{F}_2(x)$, then Theorem \ref{thm2new} extends \cite[Theorem $2.11$]{BGaniso}. Moreover, if $\mathcal{F}(x)=\mathcal{F}_p(x)$, then Theorem \ref{thm2new} extends \cite[Theorem $2.13$]{BGaniso} to the weighted anisotropic case for $w\in W_p^{s}$ and further improves the range of $p\geq 2$ to any $1<p<\infty$.
\end{Remark}



\begin{Remark}\label{ansiowgtrmk}
Both Theorem \ref{thm2} and \ref{thm2new} gives
\begin{equation}\label{anisowgtrmk2}
\begin{split}
\mu(\Omega)=\|V_{\delta}\|^p=\int_{\Omega}\mathcal{F}(\nabla V_{\delta})^p w\,dx,
\end{split}
\end{equation}
where $V_{\delta}:=\zeta_{\delta}u_{\delta}\in W_0^{1,p}(\Om,w)$ for 
$$
\zeta_{\delta}=\Bigg(\int_{\Omega}u_{\delta}^{1-\delta}f\,dx\Bigg)^{-\frac{1}{1-\delta}}.
$$
Moreover, $V_{\delta}$ satisfies
$$
\int_{\Omega}V_{\delta}^{1-\delta}f\,dx=1,
$$
and the equation
\begin{equation}\label{rmkeqn}
-\mathcal{F}_{p,w} V_{\delta}=\mu(\Omega)f V_{\delta}^{-\delta}\text{ in }\Omega,\,V_{\delta}>0 \text{ in }\Omega.
\end{equation}
\end{Remark}

Our final results concerning the problem $\mathcal{(R)}$ are stated as follows:

\begin{Theorem}\label{ethm}
Let $2\leq p<\infty$ and $w\in W_p^{s}$ for some $s\in I$. Assume that 
\begin{enumerate}
\item[(a)] $h\in L^t(\Om)\setminus\{0\}$ is nonnegative for $t>\frac{p_s^{*}}{p_s^{*}-p}$, if $1\leq p_s<N$ and
\item[(b)] $h\in L^t(\Om)\setminus\{0\}$ is nonnegative for $t>\frac{r}{r-p}$, if $p_s\geq N$ and $r>p$. 
\end{enumerate}
Then the problem $\mathcal{(R)}$ admits a weak solution $v\in W^{1,p}_{\mathrm{loc}}(\Om,w)\cap L^\infty(\Om)$.
\end{Theorem}
Moreover, we have
\begin{Theorem}\label{ethmnew}
If $w\in W_p^{s}$ for some $s\in I$ and $\mathcal{F}_{p,w}=\Delta_{p,w}$ or $\mathcal{S}_{p,w}$ as given by \eqref{exlap}, then Theorem \ref{ethm} holds, for any $1<p<\infty$.
\end{Theorem}

\begin{Remark}\label{allrmk}
If $0<c\leq w\leq d$ for some constants $c$ and $d$, then noting Remark \ref{Embrmk} our main results Theorem \ref{thm1}-\ref{thm1new}, \ref{thm2}, \ref{thm2new}, \ref{ethm}-\ref{ethmnew} holds by replacing $p_s$ with $p$.  
\end{Remark}
\section{Auxiliary results for the problem $\mathcal{(S)}$ and for the weighted anisotropic Sobolev inequality}
Throughout the rest of the article, we assume $1<p<\infty$ and $w\in W_p^{s}$ for some $s\in I$ unless otherwise mentioned. We recall that by $\|u\|$ we denote the norm $\|u\|_{W_0^{1,p}(\Om,w)}$ as defined in \eqref{equinorm}. For $n\in\mathbb{N}$, we investigate the following approximated problem 
\begin{equation}\label{mainapprox}
\begin{split}
-\mathcal{F}_{p,w}u=\frac{f_{n}}{(u^{+}+\frac{1}{n})^{\delta}}+\frac{g_{n}}{(u^{+}+\frac{1}{n})^{\gamma}}\text{ in }\Omega,\quad u=0\text{ on }\partial\Omega,
\end{split}
\end{equation}
where $f_{n}(x)=\min\{f(x),n\}$ and $g_{n}(x)=\min\{g(x),n\}$, provided $(f,g)\big(\neq(0,0)\big)\in L^{m_{\delta}}(\Omega)\times L^{r_{\gamma}}(\Omega)$ is nonnegative, where $m_\delta$ and $r_\gamma$ is given by Theorem \ref{thm1}. First we prove the following result that is important to obtain the existence and further qualitative properties of solutions for the problem \eqref{mainapprox} as shown in Lemma \ref{exisapprox}.
\begin{Lemma}\label{auxresult}
Let $2\leq p<\infty$ and $\xi\in L^{\infty}(\Om)\setminus\{0\}$ be nonnegative in $\Om$. Then there exists a unique solution $u\in W_0^{1,p}(\Om,w)\cap L^{\infty}(\Om)$ of the problem
\begin{equation}\label{auxresulteqn}
\begin{split}
-\mathcal{F}_{p,w}u=\xi\text{ in }\Om,\quad u>0\text{ in }\Om,
\end{split}
\end{equation}
such that for every $\omega\Subset\Om$, there exists a constant $c_{\omega}$, satisfying $u\geq c_{\omega}>0$ in $\omega$.
\end{Lemma}
\begin{proof}
\textbf{Existence:} To prove the existence, we define the energy functional $I:W_0^{1,p}(\Om,w)\to\mathbb{R}$ by
$$
I(u):=\frac{1}{p}\int_{\Om}\mathcal{F}(\nabla u)^p w\,dx-\int_{\Om}\xi u\,dx.
$$
Since $\xi\in L^\infty(\Om)$, using Lemma \ref{embedding}, we have
\begin{equation}\label{coercive}
\begin{split}
I(u)&\geq\frac{\|u\|^p}{p}-C|\Om|^\frac{p-1}{p}\|\xi\|_{L^\infty(\Om)}\|u\|,
\end{split}
\end{equation}
where $C>0$ is the Sobolev constant. Thus $I$ is coercive, since $p>1$. Next, we observe that $I$ is also convex. Indeed, let us define the energy functional $I_1:W_0^{1,p}(\Om,w)\to\mathbb{R}$ by
$$
I_1(u):=\frac{1}{p}\int_{\Om}\mathcal{F}(\nabla u)^p w\,dx,
$$
and $I_2:W_0^{1,p}(\Om,w)\to\mathbb{R}$ by

$$
I_2(u):=-\int_{\Om}\xi u\,dx.
$$

By the property $(D)$ of Lemma \ref{Happ}, we have $\mathcal{F}^p$ is convex and hence $I_1$ is convex. It is easy to see that $I_2$ is linear and thus $I$ is convex. Also, $I$ is a $C^1$ functional. Therefore, $I$ is weakly lower semicontinous. As a consequence of coercivity, weak lower semicontinuity and convexity, $I$ has a minimizer, say $u\in W_0^{1,p}(\Om,w)$ which solves the equation
\begin{equation}\label{auxeqn1}
-\mathcal{F}_{p,w}u=\xi\text{ in }\Om.
\end{equation}

\textbf{Uniqueness:} Let $u_1, u_2\in W_0^{1,p}(\Om,w)$ solves the problem \eqref{auxeqn1}. Therefore,
\begin{equation}\label{unif1}
\int_{\Om}w(x)\mathcal{F}(\nabla u_1)^{p-1}\nabla_{\eta}\mathcal{F}(\nabla u_1)\nabla\phi\,dx=\int_{\Om}\xi\phi\,dx,
\end{equation}
and 
\begin{equation}\label{unif22}
\int_{\Om}w(x)\mathcal{F}(\nabla u_2)^{p-1}\nabla_{\eta}\mathcal{F}(\nabla u_2)\nabla\phi\,dx=\int_{\Om}\xi\phi\,dx
\end{equation}
holds for every $\phi\in W_0^{1,p}(\Om,w)$. We choose $\phi=u_1-u_2$ and then subtracting \eqref{unif1} with \eqref{unif22}, we obtain
\begin{equation}\label{unif3}
\int_{\Om}w(x)\left\{\mathcal{F}(\nabla u_1)^{p-1}\nabla_{\eta}\mathcal{F}(\nabla u_1)-\mathcal{F}(\nabla u_2)^{p-1}\nabla_{\eta}\mathcal{F}(\nabla u_2)\right\}\nabla\big((u_1-u_2)\big)\,dx=0.
\end{equation}
By Lemma \ref{alg}, we get
$$
\int_{\Om}\big|\mathcal{F}\big(\nabla(u_1-u_2)\big)\big|^p w\,dx=0.
$$
Thus $u_1=u_2$ in $\Om$. Hence the uniqueness follows.\\
\textbf{Boundedness:} From Lemma \ref{embedding}, noting the continuity of $X\hookrightarrow L^l(\Omega)$ for some $l>p$ and then proceeding analogous to the proof of \cite[Lemma $3.1$]{BGM}, we obtain
$$
\|u\|_{L^\infty(\Omega)}\leq C,
$$
for some positive constant $C$ depending on $\|\xi\|_{L^\infty(\Om)}$. \\
\textbf{Positivity:} Choosing $u_-:=\min\{u,0\}$ as a test function in \eqref{auxeqn1} and since $\xi\geq 0$, by \eqref{lbd} we obtain
\begin{equation*}
\begin{split}
\int_{\Omega}\mathcal{F}\big(\nabla u_-\big)^p w\,dx=\int_{\Omega}\mathcal{F}(\nabla u)^{p-1}\nabla_{\eta}\mathcal{F}(\nabla u)\nabla u_- w\,dx=\int_{\Omega}\xi u_-\,dx\leq 0,
\end{split}
\end{equation*}
which gives, $u\geq 0$ in $\Omega$. Further $g\neq 0$ gives $u\neq 0$ in $\Omega$. Noting Corollary \ref{regapp}, we apply \cite[Theorem 3.59]{Juh} so that for every $\omega\Subset\Omega$, there exists a constant $c_{\omega}$ such that $u\geq c_{\omega}>0$ in $\Omega$. Thus $u>0$ in $\Omega$.
\end{proof}

\begin{Remark}\label{auxrmk}
We remark that, when $\mathcal{F}_{p,w}=\Delta_{p,w}$ or $\mathcal{S}_{p,w}$ as given by \eqref{exlap}, noting Lemma \ref{AI} and following the same proof above, Lemma \ref{auxresult} holds for any $1<p<\infty$.
\end{Remark}

\begin{Lemma}\label{exisapprox}
Let $2\leq p<\infty$. Then, 
\begin{enumerate}
\item[(A)] (Existence) for every $n\in\mathbb{N}$, the problem \eqref{mainapprox} admits a positive solution $u_{n}\in W_0^{1,p}(\Om,w)\cap L^\infty(\Om)$,
\item[(B)] (Uniqueness and monotonicity) $u_n$ is unique and $u_{n+1}\geq u_{n}$ for every $n$, 
\item[(C)] (Uniform positivity) for every $\omega\Subset\Omega$, there exists a constant $c_{\omega}$ (independent of $n$) such that $u_{n}\geq c_{\omega}>0$ in $\omega$. 
\item[(D)] (Uniform boundedness) $\|u_{n}\|\leq c$ for some positive constant $c$ independent of $n$.
\end{enumerate}
\end{Lemma}
\begin{proof}
\begin{enumerate} 
\item[(A)] Let $n\in\mathbb{N}$ be fixed. By Lemma \ref{auxresult}, for every $\zeta\in L^p(\Omega)$, there exists a unique positive solution $u\in W_0^{1,p}(\Om,w)\cap L^\infty(\Om)$ such that
\begin{equation}\label{approxfixed}
-\mathcal{F}_{p,w}u=\frac{f_{n}}{(\zeta^{+}+\frac{1}{n})^\delta}+\frac{g_{n}}{(\zeta^{+}+\frac{1}{n})^\gamma}\text{ in }\Omega.
\end{equation} 
 We define the operator $S:L^p(\Omega)\to L^p(\Omega)$ by $S(\zeta)=u$ where $u$ solves \eqref{approxfixed}. Choosing $u$ as a test function in \eqref{approxfixed} and using the property \eqref{lbd}, for some positive constant $c=c(n)$, we arrive at
\begin{align*}
\|u\|^p\leq\int_{\Omega}(n^{\delta+1}+n^{\gamma+1})u\,dx\leq c\Big(\int_{\Omega}|u|^p\,dx\Big)^\frac{1}{p}.
\end{align*}
Thus from Lemma \ref{embedding}, we have
$$
\|u\|\leq c.
$$
Now applying Schauder's fixed point theorem as in \cite[Lemma $3.2$]{BGM} the existence of a fixed point $u_{n}$ of $S$ follows. As a consequence, $u_n$ solves the problem \eqref{mainapprox}. Moreover, by Lemma \ref{auxresult}, we have $u_n>0$ in $\Om$ and for every $\omega\Subset\Om$, there exists a constant $c_{\omega}$ such that $u_1\geq c_{\omega}>0$ in $\omega$.
\item[(B)] Choosing $\phi = (u_{n}-u_{n+1})^+$ as a test function in \eqref{mainapprox} we have
\begin{equation*}
\begin{split}
J&=\langle-\mathcal{F}_{p,w}(u_{n})+\mathcal{F}_{p,w}(u_{n+1}),(u_{n}-u_{n+1})^{+}\rangle\\
&=\int_{\Omega}\Big\{\frac{f_{n}}{\big(u_{n}+\frac{1}{n}\big)^\delta}-\frac{f_{n+1}}{\big(u_{n+1}+\frac{1}{n+1}\big)^\delta}+\frac{g_{n}}{\big(u_{n}+\frac{1}{n}\big)^\gamma}-\frac{g_{n+1}}{\big(u_{n+1}+\frac{1}{n+1}\big)^\gamma}\Big\}(u_{n}-u_{n+1})^+\,dx\\
&:=J_1+J_2. 
\end{split}
\end{equation*}
Using the inequalities $f_{n}(x) \leq f_{n+1}(x)$, we obtain
\begin{align*}
J_1&=\int_{\Omega}\Big\{\frac{f_{n}}{\big(u_{n}+\frac{1}{n}\big)^\delta}-\frac{f_{n+1}}{\big(u_{n+1}+\frac{1}{n+1}\big)^\delta}\Big\}(u_{n}-u_{n+1})^+\,dx\\
&\leq\int_{\Omega}f_{n+1}\Big\{\frac{1}{\big(u_{n}+\frac{1}{n}\big)^\delta}-\frac{1}{\big(u_{n+1}+\frac{1}{n+1}\big)^\delta}\Big\}(u_{n}-u_{n+1})^+\,dx\leq 0.
\end{align*}
Similarly, using $g_{n}(x)\leq g_{n+1}(x)$, we get
$$
J_2=\int_{\Omega}\Big\{\frac{g_{n}}{\big(u_{n}+\frac{1}{n}\big)^\gamma}-\frac{g_{n+1}}{\big(u_{n+1}+\frac{1}{n+1}\big)^\gamma}\Big\}(u_{n}-u_{n+1})^+\,dx\leq 0.
$$
Hence, we have $J\leq 0$. Noting this fact and $p\geq 2$, using Lemma \ref{alg}, we have
\begin{align}\label{algapp}
\int_{\Om}\big|\mathcal{F}\big(\nabla(u_n-u_{n+1})^+\big)\big|^p w\,dx=0.
\end{align}
Therefore, $u_{n+1}\geq u_{n}$ in $\Omega$. Uniqueness follows similarly.
\item[(C)] From the above estimate in $(A)$, we know that $u_{1}\geq c_{\omega}>0$ for every $\omega\Subset\Omega$. Hence using the monotonicity, for every $\omega\Subset\Omega$, we get $u_{n}\geq c_{\omega}>0$ in $\omega$, for some positive constant $c_{\omega}$ (independent of $n$).
\item[(D)] We only consider the case $1\leq p_s<N$, since the other cases are analogous. To this end, we choose $u_{n}$ as a test function in \eqref{mainapprox} and using \eqref{lbd}, we get
\begin{equation}\label{unif}
\begin{split}
\|u_{n}\|^p&\leq\int_{\Omega}f u_{n}^{1-\delta}\,dx+\int_{\Omega}g u_{n}^{1-\gamma}\,dx\\
&\leq \|f\|_{L^{m_{\delta}}(\Omega)}\Big(\int_{\Omega}u_{n}^{(1-\delta)m_{\delta}'}\,dx\Big)^\frac{1}{m_{\delta}'}+\|g\|_{L^{r_{\gamma}}(\Omega)}\Big(\int_{\Omega}u_{n}^{(1-\gamma)r_{\gamma}'}\,dx\Big)^\frac{1}{r_{\gamma}'}\\
&=\|f\|_{L^{m_{\delta}}(\Omega)}\Big(\int_{\Omega}u_{n}^{p_s^{*}}\,dx\Big)^\frac{1-\delta}{p_s^{*}}+\|g\|_{L^{r_{\gamma}}(\Omega)}\Big(\int_{\Omega}u_{n}^{p_s^{*}}\,dx\Big)^\frac{1-\gamma}{p_s^{*}}\\
&\leq\|f\|_{L^{m_{\delta}}(\Omega)}\|u_n\|^{1-\delta}+\|g\|_{L^{r_{\gamma}}(\Omega)}\|u_n\|^{1-\gamma},
\end{split}
\end{equation}
where in the final step above, we have employed Lemma \ref{embedding}. Therefore, we have $\|u_{n}\|\leq c$, for some positive constant $c$ (independent of $n$).
\end{enumerate}
\end{proof}

\begin{Remark}\label{anisormk}
We remark that, when $\mathcal{F}_{p,w}=\Delta_{p,w}$ or $\mathcal{S}_{p,w}$ as given by \eqref{exlap}, noting Lemma \ref{AI} along with Remark \ref{auxrmk} and following the same proof above, Lemma \ref{exisapprox} holds for any $1<p<\infty$.
\end{Remark}

\begin{Remark}\label{limit}
As a consequence of Lemma \ref{exisapprox} and Remark \ref{anisormk}, let $u_{\delta,\gamma}\in W_0^{1,p}(\Om,w)$ be the weak and pointwise limit of $u_{n}$. Then using the monotonicity property from $(B)$ in Lemma \ref{exisapprox}, it follows that $u_{n}\leq u_{\delta,\gamma}$ for all $n\in\mathbb{N}$. Below, we observe that $u_{\delta,\gamma}$ is our required solution.
\end{Remark}

\begin{Lemma}\label{essbdd}
Let $2\leq p<\infty$ and suppose that $(f,g)\neq (0,0)$ is nonnegative such that

\begin{enumerate}
\item[(E)] $f,g\in L^q(\Om)$ for $q>\frac{p_s^{*}}{p_s^{*}-p}$, when $1\leq p_s<N$,
\item[(F)] $f,g\in L^q(\Om)$ for $q>\frac{r}{r-p}$, when $p_s=N$ and $r>p$,
\item[(G)] $f,g\in L^1(\Om)$ for $p_s>N$. 
\end{enumerate} 
Then $\|u_n\|_{L^\infty(\Omega)}\leq C$, for some positive constant $C$ independent of $n$.
\end{Lemma}
\begin{proof}
We only prove the result under the hypothesis in $(E)$, since the other cases are analogous. To this end, let $1\leq p_s<N$ and $f,g\in L^q(\Om)$ for $q>\frac{p_s^{*}}{p_s^{*}-p}$. Assume $k\geq 1$ and define $A(k)=\{x\in\Om:u_n(x)\geq k\}$. Choosing $\phi_k(x)=(u_n-k)^{+}$ as a test function in \eqref{mainapprox}, first using H\"older's inequality with the exponents $p_s^{*'}, p_s^*$ and then, by Young's inequality with exponents $p$ and $p'$, we obtain
\begin{equation}\label{unibdd}
\begin{split}
\|\phi_k\|^p&=\int_{\Om}\mathcal{F}(\nabla\phi_k)^p w\,dx\\
&=\int_{\Om}\frac{f_n}{\big(u_n+\frac{1}{n}\big)^\delta}\phi_k\,dx+\int_{\Om}\frac{g_n}{\big(u_n+\frac{1}{n}\big)^\delta}\phi_k\,dx\\
&\leq\int_{A(k)}f(x)\phi_k\,dx+\int_{A(k)}g(x)\phi_k\,dx\\
&\leq\left(\int_{A(k)}f^{p_s^{*'}}\,dx\right)^\frac{1}{p_s^{*'}}\left(\int_{\Om}\phi_k^{p_{s}^*}\,dx\right)^\frac{1}{p_{s}^*}+\left(\int_{A(k)}g^{p_s^{*'}}\,dx\right)^\frac{1}{p_s^{*'}}\left(\int_{\Om}\phi_k^{p_{s}^*}\,dx\right)^\frac{1}{p_{s}^*}\\
&\leq C\left(\int_{A(k)}f^{p_{s}^*{'}}\,dx\right)^\frac{1}{p_s^{*'}}\|\phi_k\|+C\left(\int_{A(k)}g^{p_{s}^*{'}}\,dx\right)^\frac{1}{p_s^{*'}}\|\phi_k\|\\
&\leq \epsilon\|\phi_k\|^p+C(\epsilon)\left(\int_{A(k)}f^{p_s^{*'}}\,dx\right)^\frac{p'}{p_s^{*'}}+C(\epsilon)\left(\int_{A(k)}g^{p_s^{*'}}\,dx\right)^\frac{p'}{p_s^{*'}}.
\end{split}
\end{equation}
Here, $C$ is the Sobolev constant from Lemma \ref{embedding} and  $C(\epsilon)>0$ is some constant depending on $\epsilon\in(0,1)$ but are independent of $n$. Note that $q>\frac{p_s^*}{p_s^*-p}$ gives $q>p_s^{*'}$. Therefore, fixing $\epsilon\in(0,1)$ and again using H\"older's inequality with exponents $\frac{q}{p_s^{*'}}$ and $\big(\frac{q}{p_s^{*'}}\big)'$, for some constant $C>0$ which is independent of $n$, we obtain
\begin{equation}\label{supbd1}
\begin{split}
\|\phi_k\|^p&\leq C\left(\int_{A(k)}f^{p_s^{*'}}\,dx\right)^\frac{p'}{p_s^{*'}}+C\left(\int_{A(k)}g^{p_s^{*'}}\,dx\right)^\frac{p'}{p_s^{*'}}\\
&\leq C\Bigg\{\left(\int_{\Omega}f^q\,dx\right)^\frac{p'}{q}+\left(\int_{\Omega}g^q\,dx\right)^\frac{p'}{q}\Bigg\}|A(k)|^{\frac{p'}{p_s^{*'}}\frac{1}{\big(\frac{q}{p_s^{*'}}\big)'}}\\
&\leq C|A(k)|^{\frac{p'}{p_s^{*'}}\frac{1}{\big(\frac{q}{p_s^{*'}}\big)'}}.
\end{split}
\end{equation}
Let $h>0$ be such that $1\leq k<h$. Then, $A(h)\subset A(k)$ and for any $x\in A(h)$, we have $u_n(x)\geq h$. So, $u_n(x)-k\geq h-k$ in $A(h)$. Combining these facts along with \eqref{supbd1} and again using Lemma \ref{embedding} for some constant $C>0$ (independent of $n$), we arrive at the estimate below
\begin{align*}
(h-k)^p|A(h)|^\frac{p}{p_s^*}&\leq\left(\int_{A(h)}(u_n-k)^{p_s^*}\,dx\right)^\frac{p}{p_s^{*}}\\
&\leq\left(\int_{A(k)}(u_n-k)^{p_s^*}\,dx\right)^\frac{p}{p_s^{*}}\leq C\|\phi_k\|^p\leq C|A(k)|^{\frac{p'}{p_s^{*'}}\frac{1}{\big(\frac{q}{p_s^{*'}}\big)'}}.
\end{align*}
Thus, for some constant $C>0$ (independent of $n$), we have
$$
|A(h)|\leq \frac{C}{(h-k)^{p_s^*}}|A(k)|^{\alpha},
$$
where
$$
\alpha={\frac{p_s^{*}p'}{pp_s^{*'}}\frac{1}{\big(\frac{q}{p_s^{*'}}\big)'}}.
$$
Due to the assumption, $q>\frac{p_s^*}{p_s^*-p}$, we have $\alpha>1$. Hence, by \cite[Lemma B.1]{Stam}, we have
$$
||u_n||_{L^\infty(\Om)}\leq C,
$$
for some positive constant $C>0$ independent of $n$.
\end{proof}
We end this section by establishing the following properties of $u_n$ that are very important to prove the weighted anisotropic Sobolev inequality.
\begin{Lemma}\label{lemma1}
The solutions $u_n$ of the problem \eqref{mainapprox} found in Lemma \ref{exisapprox} has the following properties:
\begin{enumerate}
\item[(a)] Let $n\in\mathbb{N}$ and $\phi\in W_0^{1,p}(\Om,w)$. Then we have
\begin{equation}\label{prop1}
\|u_{n}\|^{p}\leq\|\phi\|^{p}+p\int_{\Omega}\frac{(u_{n}-\phi)}{(u_{n}+\frac{1}{n})^\delta}f_{n}\,dx+p\int_{\Omega}\frac{(u_{n}-\phi)}{(u_{n}+\frac{1}{n})^\gamma}g_{n}\,dx.
\end{equation}
\item[(b)] (Monotonicity in norm) For every $n\in\mathbb{N}$, we have $\|u_{n}\|\leq\|u_{n+1}\|$.
\item[(c)] (Strong convergence) Upto a subsequence $\{u_n\}$ converges strongly to $u_{\delta,\gamma}$ in $W_0^{1,p}(\Om,w)$.
\item[(d)] Further, $u_{\delta,\gamma}$ is a minimizer of the energy functional $I_{\delta,\gamma}:W_0^{1,p}(\Om,w)\to\mathbb{R}$ be defined by
\begin{align}\label{eng}
I_{\delta,\gamma}(v):=\frac{1}{p}\|v\|^p-\frac{1}{1-\delta}\int_{\Omega}(v^{+})^{1-\delta}f\,dx-\frac{1}{1-\delta}\int_{\Omega}(v^{+})^{1-\gamma}g\,dx.
\end{align}
\end{enumerate}
\end{Lemma}
\begin{proof}
\begin{enumerate}
\item[(a)] Let us fix $\xi\in W_0^{1,p}(\Om,w)$. Then by Lemma \ref{auxresult}, there exists a unique solution $v\in W_0^{1,p}(\Om,w)$ of the problem
\begin{equation}\label{auxeqn}
-\mathcal{F}_{p,w}v=\frac{f_{n}(x)}{(\xi^{+}+\frac{1}{n})^\delta}+\frac{g_{n}(x)}{(\xi^{+}+\frac{1}{n})^\gamma},\,v>0\text{ in }\Omega,\,v=0\text{ on }\partial\Omega.
\end{equation}

Also $v$ is a minimizer of the functional $J:W_0^{1,p}(\Om,w)\to\mathbb{R}$ given by
$$
J(\phi):=\frac{1}{p}\|\phi\|^{p}-\int_{\Omega}\frac{f_{n}}{(\xi^{+}+\frac{1}{n})^\delta}\phi\,dx-\int_{\Omega}\frac{g_{n}}{(\xi^{+}+\frac{1}{n})^\gamma}\phi\,dx.
$$

Therefore, for every $\phi\in W_0^{1,p}(\Om,w)$, we have $J(v)\leq J(\phi)$ which gives
\begin{equation}\label{mineqn}
\begin{split}
&\frac{1}{p}\|v\|^{p}-\int_{\Omega}\frac{f_{n}}{(\xi^{+}+\frac{1}{n})^\delta}v\,dx-\int_{\Omega}\frac{g_{n}}{(\xi^{+}+\frac{1}{n})^\gamma}v\,dx\\
&\quad\leq \frac{1}{p}\|\phi\|^{p}-\int_{\Omega}\frac{f_{n}}{(\xi^{+}+\frac{1}{n})^\delta}\phi\,dx-\int_{\Omega}\frac{g_{n}}{(\xi^{+}+\frac{1}{n})^\gamma}\phi\,dx.
\end{split}
\end{equation}
Setting $v=\xi=u_{n}$ in the inequality \eqref{mineqn}, the estimate \eqref{prop1} follows. 

\item[(b)] By Lemma \ref{exisapprox}, we have $u_{n}\leq u_{n+1}$. Then choosing $\phi=u_{n+1}$ in \eqref{prop1}, we obtain $\|u_{n}\|\leq \|u_{n+1}\|$.

\item[(c)] We choose $\phi=u_{\delta,\gamma}$ in \eqref{prop1} and then using the property $u_{n}\leq u_{\delta,\gamma}$ from Remark \ref{limit}, we have $\|u_{n}\|\leq \|u_{\delta,\gamma}\|$. Hence using the norm monotonicity property $\|u_n\|\leq \|u_{n+1}\|$ from $(b)$, we have
\begin{equation}\label{lim1}
\lim_{n\to\infty}\|u_{n}\|\leq \|u_{\delta,\gamma}\|.
\end{equation}
Moreover since $u_{n}\rightharpoonup u_{\delta,\gamma}$ weakly in $W_0^{1,p}(\Om,w)$, we get
\begin{equation}\label{lim2}
\|u_{\delta,\gamma}\|\leq \lim_{n\to\infty}\|u_{n}\|.
\end{equation}
Thus from \eqref{lim1} and \eqref{lim2}, the result follows.

\item[(d)] It is enough to show that
\begin{equation}\label{limclm}
I_{\delta,\gamma}(u_{\delta,\gamma})\leq I_{\delta,\gamma}(v),\quad\forall \,v\in W_0^{1,p}(\Om,w).
\end{equation}
Let us define the auxiliary functional $I_{n}:W_0^{1,p}(\Om,w)\to\mathbb{R}$ by
$$
I_{n}(v):=\frac{1}{p}\|v\|^{p}-\int_{\Omega}G_n(v)f_{n}\,dx-\int_{\Omega}H_n(v)g_{n}\,dx,
$$
where
$$
G_n(t):=\frac{1}{1-\delta}\Big(t^{+}+\frac{1}{n}\Big)^{1-\delta}-\Big(\frac{1}{n}\Big)^{-\delta}t^{-},
$$
and
$$
H_n(t):=\frac{1}{1-\gamma}\Big(t^{+}+\frac{1}{n}\Big)^{1-\gamma}-\Big(\frac{1}{n}\Big)^{-\gamma}t^{-}.
$$
Then we observe that $I_{n}$ is $C^1$, bounded below and coercive. As a consequence, $I_{n}$ has a minimizer at some $v_{n}\in W_0^{1,p}(\Om,w)$. Therefore, it follows that $I_n(v_n)\leq I_n(v_n^{+})$, which gives $v_{n}\geq 0$ in $\Omega$. Noting that $<I_{n}{'}(v_{n}),\phi>=0$ for all $\phi\in W_0^{1,p}(\Om,w)$, we conclude that $v_{n}$ solves \eqref{mainapprox}. By the uniqueness property from Lemma \ref{exisapprox}, we have $u_{n}=v_{n}$. Hence $u_{n}$ is a minimizer of $I_{n}$. Therefore, we obtain
\begin{equation}\label{Inlim}
I_n(u_n)\leq I_n(v+),\quad\forall\,v\in W_0^{1,p}(\Om,w).
\end{equation}
Now we pass the limit as $n\to\infty$ in \eqref{Inlim} to prove our claim \eqref{limclm}.
Firstly, by Remark \ref{limit} using the fact $u_{n}\leq u_{\delta,\gamma}$ along with the Lebesgue dominated convergence theorem, we have
\begin{equation}\label{Inlim1}
\begin{split}
\lim_{n\to\infty}\int_{\Omega}G_n(u_{n})f_{n}\,dx&=\frac{1}{1-\delta}\int_{\Omega}(u_{\delta,\gamma})^{1-\delta}f\,dx,\\
\lim_{n\to\infty}\int_{\Omega}H_n(u_{n})g_{n}\,dx&=\frac{1}{1-\gamma}\int_{\Omega}(u_{\delta,\gamma})^{1-\gamma}g\,dx.
\end{split}
\end{equation}
Furthermore, by the strong convergence property $(c)$ above, we have 
\begin{equation}\label{Inlim2}
\lim_{n\to\infty}\|u_{n}\|= \|u_{\delta,\gamma}\|.
\end{equation}
Hence, using \eqref{Inlim1} and \eqref{Inlim2}, we have
\begin{equation}\label{Inlim3}
\lim_{n\to\infty}I_{n}(u_{n})=I_{\delta,\gamma}(u_{\delta,\gamma}).
\end{equation}
Moreover, for any $v\in W_0^{1,p}(\Om,w)$, we have
\begin{equation}\label{Jnlim1}
\begin{split}
\lim_{n\to\infty}\int_{\Omega}G_n(v^{+})f_{n}\,dx&=\frac{1}{1-\delta}\int_{\Omega}(v^{+})^{1-\delta}f\,dx,\\
\lim_{n\to\infty}\int_{\Omega}H_n(v^{+})g_{n}\,dx&=\frac{1}{1-\gamma}\int_{\Omega}(v^{+})^{1-\delta}g\,dx.
\end{split}
\end{equation}
 Finally, letting $n\to\infty$ in \eqref{Inlim} and then employing the estimates \eqref{Inlim3}, \eqref{Jnlim1} and that $\|v^{+}\|\leq \|v\|$, the claim \eqref{limclm} follows.
\end{enumerate}
\end{proof}

\section{Auxiliary results for the problem $\mathcal{(R)}$}
This section deals to establish some results that are very crucial to prove Theorem \ref{ethm}-\ref{ethmnew}. In order to proceed, let $h\in L^t(\Om)\setminus\{0\}$ be nonnegative as given by Theorem \ref{ethm} and we consider the following approximated problem for every $n\in\mathbb{N}$, given by
\begin{equation}\label{eapprox}
-\mathcal{F}_{p,w}u=h_n(x)e^\frac{1}{\big(u^++\frac{1}{n}\big)}\text{ in }\Om,\,u=0\text{ on }\partial\Om,
\end{equation}
where $h_n(x)=\min\{h(x),n\}$.

Then, we have the following results for the problem \eqref{eapprox}.

\begin{Lemma}\label{eexisapprox}
Let $2\leq p<\infty$. Then, 
\begin{enumerate}
\item[(A)] for every $n\in\mathbb{N}$, the problem \eqref{eapprox} admits a positive solution $v_{n}\in W_0^{1,p}(\Om,w)$,
\item[(B)] $v_n$ is unique and $v_{n+1}\geq v_{n}$ for every $n$,
\item[(C)] for every $\omega\Subset\Omega$, there exists a constant $c_{\omega}$ (independent of $n$) satisfying $v_{n}\geq c_{\omega}>0$ in $\omega$,
\item[(D)] $\|v_{n}\|_{L^\infty(\Om)}\leq c$ for some positive constant $c$ independent of $n$.
\end{enumerate}
\end{Lemma}

\begin{proof}
\begin{enumerate}
\item[(A)] Let us fix $n\in\mathbb{N}$, $v\in L^p(\Om)$ and denote by
$$
H_n=h_n e^\frac{1}{\big(v^++\frac{1}{n}\big)}.
$$
Then, noting the fact that
$$
H_n\in L^\infty(\Om)\text{ and }H_n {v_n}_-\leq 0,
$$
the existence of a positive solution $v_n\in W_0^{1,p}(\Om,w)$ of the problem \eqref{eapprox} follows by the same arguments as in the proof of Lemma \ref{exisapprox}. 
\item[(B)] Observing that
$$
(H_n-H_{n+1})(v_n-v_{n+1})^+\leq 0,
$$
and then proceeding similarly as in the proof of Lemma \ref{exisapprox}, we obtain the monotonicity of $v_n$. The uniqueness is analogous. 

\item[(C)] We observe that $v_n\neq 0$ in $\Om$ and note that $v_{n}\geq v_1$ from $(B)$. Then proceeding analogous to the proof of Lemma \ref{exisapprox}, for every $\omega\Subset\Om$, we have $v_n\geq c_{\omega}>0$ for some constant $c_{\omega}>0$ that is independent of $n$.

\item[(D)] Now the uniform boundedness of $\{v_n\}$ in $L^\infty(\Om)$ follows analogously in the proof of Lemma \ref{essbdd}.
\end{enumerate}
\end{proof}

\begin{Remark}\label{ermk}
If $\mathcal{F}_{p,w}=\Delta_{p,w}$ or $\mathcal{S}_{p,w}$ as given by \eqref{exlap}, noting Lemma \ref{AI} and following the exact arguments in the proof above, Lemma \ref{eexisapprox} holds for any $1<p<\infty$. 
\end{Remark}
\section{Proof of the main results}
\textbf{Proof of Theorem \ref{thm1}:}
\begin{enumerate}
\item[(a)] Let $u,v\in W_0^{1,p}(\Om,w)$ be weak solutions of $\mathcal{(S)}$. Then arguing similarly as in the proof of \cite[Lemma $2.13$]{BGM}, we choose $\phi=(u-v)^{+}\in W_0^{1,p}(\Om,w)$ as a test function in \eqref{weakform1} and obtain
\begin{equation}\label{uni1}
\int_{\Om}w(x)\mathcal{F}(\nabla u)^{p-1}\nabla_{\eta}\mathcal{F}(\nabla u)\nabla(u-v)^+\,dx=\int_{\Om}\Big(\frac{f}{u^\delta}+\frac{g}{u^\gamma}\Big)(u-v)^+\,dx,
\end{equation}
\begin{equation}\label{uni2}
\int_{\Om}w(x)\mathcal{F}(\nabla v)^{p-1}\nabla_{\eta}\mathcal{F}(\nabla v)\nabla(u-v)^+\,dx=\int_{\Om}\Big(\frac{f}{v^\delta}+\frac{g}{v^\gamma}\Big)(u-v)^+\,dx.
\end{equation}
Subtracting \eqref{uni1} and \eqref{uni2}, we have
\begin{align*}
&\int_{\Om}w(x)\{\mathcal{F}(\nabla u)^{p-1}\nabla_{\eta}\mathcal{F}(\nabla u)-\mathcal{F}(\nabla v)^{p-1}\nabla_{\eta}\mathcal{F}(\nabla v)\}\nabla(u-v)^+\,dx\\
&=\int_{\Omega}\Big(\big(\frac{1}{u^{\delta}}-\frac{1}{v^{\delta}}\big)f+\big(\frac{1}{u^{\gamma}}-\frac{1}{v^\gamma}\big)g\Big)(u-v)^{+}\,dx\leq 0.
\end{align*}
Thus, applying Lemma \ref{alg}, we obtain 
$$
\int_{\Om}\mathcal{F}\big(\nabla(u-v)^+\big)^p w\,dx=0,
$$
which gives $u\leq v$ in $\Omega$. In a similar way, we have $v\leq u$ in $\Omega$. Hence the result follows.

\item[(b)] By Lemma \ref{exisapprox}, for every $n\in\mathbb{N}$, there exists $u_{n}\in W_0^{1,p}(\Om,w)$ such that
\begin{equation}\label{exislim1}
\int_{\Om}w(x)\mathcal{F}(\nabla u_n)^{p-1}\nabla_{\eta}\mathcal{F}(\nabla u_n)\nabla\phi\,dx=\int_{\Omega}\Big(\frac{f_{n}}{u_{n}^{\delta}}+\frac{g_{n}}{u_{n}^{\gamma}}\Big)\phi\,dx,\quad\forall\,\phi\in C_c^{1}(\Om).
\end{equation}
\subsection*{Passing to the limit} By the strong convergence property $(c)$ in Lemma \ref{lemma1} along with Lemma \ref{embedding}, upto a subsequence, we have $\nabla u_{n}\to \nabla u_{\delta,\gamma}$ pointwise almost everywhere in $\Omega$. Therefore, for every $\phi\in C_c^{1}(\Omega)$, it holds that
\begin{equation}\label{exislim4}
\begin{split}
\lim_{n\to\infty}\int_{\Omega}w\mathcal{F}(\nabla u_n)^{p-1}\nabla_{\eta}\mathcal{F}(\nabla u_n)\nabla\phi\,dx&=\int_{\Omega}w(x)\mathcal{F}(\nabla u_{\delta,\gamma})^{p-1}\nabla_{\eta}\mathcal{F}(\nabla u_{\delta,\gamma})\nabla\phi\,dx.
\end{split}
\end{equation}

Denote by $\mathrm{supp}\,\phi=\omega\Subset \Om$ and thus by Lemma \ref{exisapprox}, there exists a constant $c_{\omega}>0$ that is independent of $n$ such that $u_n\geq c_{\omega}>0$ in $\omega$. Hence, we get 
\begin{equation*}\label{exislim2}
\Big(\frac{f_{n}}{u_{n}^{\delta}}+\frac{g_{n}}{u_{n}^{\gamma}}\Big)\phi\leq\|\phi\|_{L^{\infty}(\Omega)}\Big(\frac{f}{c_{\omega}^{\delta}}+\frac{g}{c_{\omega}^{\gamma}}\Big)\in L^1(\Omega).
\end{equation*}
Since $u_n\to u_{\delta,\gamma}$ pointwise alomost everywhere in $\Omega$, as an application of the Lebesgue dominated convergence theorem, we deduce that
\begin{equation}\label{exislim3} \lim_{n\to\infty}\int_{\Omega}\Big(\frac{f_{n}}{u_{n}^{\delta}}+\frac{g_{n}}{u_{n}^{\gamma}}\Big)\phi\,dx=\int_{\Omega}\Big(\frac{f}{u_{\delta,\gamma}^{\delta}}+\frac{g}{u_{\delta,\gamma}^{\gamma}}\Big)\phi\,dx.
\end{equation}
Combining the estimates \eqref{exislim4} and \eqref{exislim3} in \eqref{exislim1}, the result follows. \qed
\end{enumerate} \qed\\
\textbf{Proof of Theorem \ref{regthm1}:} The proof follows from Lemma \ref{essbdd}. \qed\\
\textbf{Proof of Theorem \ref{thm1new}:} Noting Lemma \ref{AI} and Remark \ref{anisormk}, following the lines of the proof of Theorem \ref{thm1}-\ref{regthm1}, the result follows.\\
\textbf{Proof of Theorem \ref{thm2}:}
\begin{enumerate}
\item[(a)] Let us set $$
S_{\delta}:=\Bigg\{v\in W_0^{1,p}(\Om,w):\int_{\Omega}|v|^{1-\delta}f\,dx=1\Bigg\}.
$$
Then it is enough to prove that
\begin{align*}
\mu(\Omega)&:=\inf_{v\in S_{\delta}}\|v\|^{p}=\|u_{\delta}\|^{\frac{p(1-\delta-p)}{1-\delta}}.
\end{align*}
It is easy to verify that $V_{\delta}=\zeta_{\delta} u_{\delta}\in S_{\delta}$, where
$$
\zeta_{\delta}=\Bigg(\int_{\Omega}u_{\delta}^{1-\delta}f\,dx\Bigg)^{-\frac{1}{1-\delta}}.
$$
Following the proof of \cite[Lemma $2.13$]{BGM}, we choose $\phi=u_{\delta}\in W_0^{1,p}(\Om,w)$ as a test function in $\eqref{weakform1}$ and noting the property \eqref{lbd}, we have
\begin{equation}\label{useful1}
\begin{split}
\int_{\Om}\mathcal{F}(\nabla u_\delta)^p w\,dx=\|u_{\delta}\|^p=\int_{\Omega}u_{\delta}^{1-\delta}f\,dx.
\end{split}
\end{equation}
First, using the homogeneity property $(H2)$ and then, by \eqref{useful1}, we have
\begin{equation}\label{extremal}
\begin{split}
\|V_{\delta}\|^{p}&=\int_{\Omega}\mathcal{F}(\nabla V_\delta)^pw\,dx=\zeta_\delta^p\int_{\Omega}\mathcal{F}(\nabla u_\delta)^pw\,dx=\|u_{\delta}\|^\frac{p(1-\delta-p)}{1-\delta}.
\end{split}
\end{equation}
Let $v\in S_{\delta}$ and define by $\mu=\|v\|^{-\frac{p}{p+\delta-1}}$. Then by Lemma \ref{lemma1}, since $u_{\delta}$ minimizes the functional $I_{\delta,\gamma}$ given by \eqref{eng}, we have 
\begin{equation}\label{m1}
I_{\delta,\gamma}(u_\delta)\leq I_{\delta,\gamma}(\mu|v|).
\end{equation}
Using \eqref{useful1}, we have
\begin{equation}\label{m2}
\begin{split}
I_{\delta,\gamma}(u_\delta)=\frac{1}{p}\|u_\delta\|^p-\frac{1}{1-\delta}\int_{\Omega}u_\delta^{1-\delta}f\,dx=\Big(\frac{1}{p}-\frac{1}{1-\delta}\Big)\|u_{\delta}\|^{p}.
\end{split}
\end{equation}
On the other hand, since $v\in S_\delta$, we have
\begin{equation}\label{m3}
\begin{split}
I_{\delta,\gamma}(\mu|v|)&=\frac{\mu^{p}}{p}\||v|\|^{p}-\frac{\mu^{1-\delta}}{1-\delta}\leq\frac{\mu^{p}}{p}\|v\|^{p}-\frac{\mu^{1-\delta}}{1-\delta}=\Big(\frac{1}{p}-\frac{1}{1-\delta}\Big)\|v\|^\frac{p(\delta-1)}{\delta+p-1}.
\end{split}
\end{equation}
Since $v\in S_{\delta}$ is arbitrary, using \eqref{m2} and \eqref{m3} in \eqref{m1}, we arrive at
\begin{equation}\label{m4}
\|u_{\delta}\|^\frac{p(1-\delta-p)}{1-\delta}\leq \inf_{v\in S_\delta}\|v\|^{p}.
\end{equation}
 Using \eqref{extremal} and \eqref{m4}, we obtain 
\begin{equation}\label{infprop}
\|V_\delta\|^p=\|u_{\delta}\|^\frac{p(1-\delta-p)}{1-\delta}\leq\inf_{v\in S_{\delta}}\|v\|^{p}.
\end{equation}
Since $V_{\delta}\in S_{\delta}$, from \eqref{infprop}, the result follows.
\item[(b)] Let \eqref{inequality1} holds. If $C>\mu(\Omega)$, then from $(a)$ above and \eqref{extremal}, we obtain
\begin{equation}\label{con1}
C\Big(\int_{\Omega}V_\delta^{1-\delta}f\,dx\Big)^\frac{p}{1-\delta}>\int_{\Om}\mathcal{F}(\nabla V_\delta)^p w\,dx.
\end{equation}
Since $V_\delta\in W_0^{1,p}(\Om,w)$, \eqref{con1} violates the hypothesis \eqref{inequality1}. Conversely, assume that
$$
C\leq \mu(\Omega)=\inf_{v\in S_{\delta}}\|v\|^{p}\leq \|V\|^{p},
$$ for all $V\in S_{\delta}$. We observe that the claim directly follows if $v=0$. So we deal with the case when $v\in W_0^{1,p}(\Om,w)\setminus\{0\}$ which gives
$$
V=\Bigg(\int_{\Omega}|v|^{1-\delta}f\,dx\Bigg)^{-\frac{1}{1-\delta}}v\in S_{\delta}.
$$
Therefore, we have
$$
C\leq\Bigg(\int_{\Omega}|v|^{1-\delta}f\,dx\Bigg)^{-\frac{p}{1-\delta}}\|v\|^p. 
$$
Hence, the result follows. \qed
\end{enumerate}  
\textbf{Proof of Theorem \ref{thm2new}:} Noting Lemma \ref{AI} and Remark \ref{anisormk}, following the lines of the proof of Theorem \ref{thm2}, the result follows. \qed

\textbf{Proof of Theorem \ref{ethm}:}
Let $\Om=\cup_{l}\Om_l$ where $\Om_l\Subset\Om_{l+1}$ are open subsets for each $l$. By Lemma \ref{eexisapprox}, we have $v_n\in L^\infty(\Om)$ and $c_l=\inf_{\Om_l}v_n>0$. Then choosing $\phi=(v_n-c_1)^+$ as a test function in \eqref{eapprox} and using \eqref{lbd}, we obtain
\begin{equation}\label{etst1}
\begin{split}
\|(v_n-c_1)^+\|^p&=\int_{\Om}w(x)\mathcal{F}(\nabla v_n)^{p-1}\nabla_{\eta}\mathcal{F}(\nabla v_n)\nabla(v_n-c_1)^+\,dx\\
&=\int_{\Om}h_n e^\frac{1}{\big(v_n+\frac{1}{n}\big)}(v_n-c_1)^+\,dx\\
&\leq e^\frac{1}{c_1}\|h\|_{L^1(\Om)}\|(v_n-c_1)^+\|_{L^\infty(\Om)}\\
&\leq ce^\frac{1}{c_1}\|h\|_{L^1(\Om)}\|(v_n-c_1)^+\|, 
\end{split}
\end{equation}
for some constant $c>0$ independent of $n$, where in the final step above, we have used the embedding result Lemma \ref{embedding}. Therefore, we obtain from \eqref{etst1} that $\{v_n\}$ is uniformly bounded in $W^{1,p}(\Om_1,w)$ and thus by Lemma \ref{embedding}, there exists $v_{\Om_1}\in W^{1,p}(\Om_1,w)$ such that upto a subsequence $\{v_{n_k^{1}}\}$ converges weakly in $W^{1,p}(\Om_1,w)$, strongly in $L^p(\Om_1)$ and almost everywhere in $\Om_1$ to $v_{\Om_1}$, say. Proceeding by induction argument, for every $l$, upto a subsequence $\{v_{n_k^{l}}\}$ of $\{v_n\}$, there exists $v_{\Om_l}$ such that $\{v_{n_k^{l}}\}$ converges weakly in $W^{1,p}(\Om_l,w)$, strongly in $L^{p}(\Om_l)$ and almost everywhere in $\Om_l$ to $v_{\Om_l}$. Let $\{v_{n_k^{l+1}}\}$ be a subsequence of $\{v_{n_k^{l}}\}$ for every $l$, where $n_l^{l}\to\infty$ as $l\to\infty$. Hence, $v_{\Om_{l+1}}=v_{\Om_l}$ in $\Om_l$ and we define by $v=v_{\Om_1}$ in $\Om_1$, $v=v_{\Om_{l+1}}$ in $\Om_{l+1}\setminus\Om_l$ for each $l$. Therefore, $v\in W^{1,p}_{\mathrm{loc}}(\Om,w)$ and also lie in $L^\infty(\Om)$ due to the property $(D)$ from Lemma \ref{eexisapprox}. As a consequence of our definition, the diagonal subsequence $\{v_{n_l}\}:=\{v_{n_l}^l\}$ satisfies
\begin{equation}\label{ecgnce2}
\begin{split}
v_{n_l}&\to v\text{ in }W^{1,p}_{\mathrm{loc}}(\Om_l,w),\\
v_{n_l}&\to v\text{ in }L^{p}_{\mathrm{loc}}(\Om_l),\\
v_{n_l}&\to v\text{ almost everywhere in }\Om_l.
\end{split}
\end{equation}
We claim that $\{v_{n_l}\}$ converges strongly to $v$ in $W^{1,p}_{\mathrm{loc}}(\Om_l,w)$. To see this, let $\omega\Subset\Om$ and $\phi\in C_c^{\infty}(\Om)$ be such that $0\leq\phi\leq 1$ in $\Om$, $\phi\equiv 1$ in $\omega$ and suppose $l_1\geq 1$ such that $\Om':=\mathrm{supp}\,\phi\subset\Om_{l_1}$. Then, by a direct computation for every $l,m\geq 1$, we have
\begin{equation}\label{estlim}
\begin{split}
&\int_{\omega}w(x)\{\mathcal{F}(\nabla v_{n_l})^{p-1}\nabla_{\eta}\mathcal{F}(\nabla v_{n_l})-\mathcal{F}(\nabla v_{n_m})^{p-1}\nabla_{\eta}\mathcal{F}(\nabla v_{n_m})\}\nabla(v_{n_l}-v_{n_m})\,dx\\
&\leq\int_{\Om}w(x)\{\mathcal{F}(\nabla v_{n_l})^{p-1}\nabla_{\eta}\mathcal{F}(\nabla v_{n_l})-\mathcal{F}(\nabla v_{n_m})^{p-1}\nabla_{\eta}\mathcal{F}(\nabla v_{n_m})\}\nabla\big(\phi(v_{n_l}-v_{n_m})\big)\,dx\\
&\quad-\int_{\Om_{l_1}}w(x)(v_{n_l}-v_{n_m})\{\mathcal{F}(\nabla v_{n_l})^{p-1}\nabla_{\eta}\mathcal{F}(\nabla v_{n_l})-\mathcal{F}(\nabla v_{n_m})^{p-1}\nabla_{\eta}\mathcal{F}(\nabla v_{n_m})\}\nabla\phi\,dx\\
&:=I-J.
\end{split}
\end{equation}
\textbf{Estimate of $I$:} We choose $\phi(v_{n_l}-v_{n_m})$ as a test function in \eqref{eapprox} and obtain for $j=l,m$ that
\begin{align*}
&\int_{\Om}w(x)\mathcal{F}(\nabla v_{n_j})^{p-1}\nabla_{\eta}\mathcal{F}(\nabla v_{n_j})\nabla\big(\phi(v_{n_l}-v_{n_m})\big)\,dx\\
&\leq\int_{\Om'}h_n e^\frac{1}{\big(v_{n_j}+\frac{1}{j}\big)}|v_{n_l}-v_{n_m}|\,dx\\
&\leq c\|h\|_{L^1(\Om)}\|v_{n_l}-v_{n_m}\|_{L^\infty(\Om')}\\
&\leq c\|h\|_{L^1(\Om)}\|v_{n_l}-v_{n_m}\|_{L^p(\Om')},
\end{align*}
for a constant $c$ independent of $l,m$. Therefore, using \eqref{ecgnce2}, the last quantity in the above estimate goes to zero as $l,m\to\infty$. As a consequence, we arrive at 
$$
I=\int_{\Om}w(x)\{\mathcal{F}(\nabla v_{n_l})^{p-1}\nabla_{\eta}\mathcal{F}(\nabla v_{n_l})-\mathcal{F}(\nabla v_{n_m})^{p-1}\nabla_{\eta}\mathcal{F}(\nabla v_{n_m})\}\nabla\big(\phi(v_{n_l}-v_{n_m})\big)\,dx\to 0
$$
as $l,m\to\infty$.\\
\textbf{Estimate of $J$:} Using \eqref{ubd} and H\"older's inequality, we get
\begin{align*}
&\int_{\Om_{l_1}}w(x)(v_{n_l}-v_{n_m})\mathcal{F}(\nabla v_{n_l})^{p-1}\nabla_{\eta}\mathcal{F}(\nabla v_{n_l})\nabla\phi\,dx\\
&\leq C_2\|\nabla\phi\|_{L^\infty(\Om)}\left(\int_{\Om'}w|\nabla v_{n_l}|^p\,dx\right)^\frac{p-1}{p}\left(\int_{\Om'}w|v_{n_l}-v_{n_m}|^p\,dx\right)^\frac{1}{p}.
\end{align*}
Noting \eqref{ecgnce2} and \cite[Theorem $2.14$]{Mikko}, the last quantity in the above estimate goes to zero as $l,m\to\infty$. Therefore, 
$$
J=\int_{\Om_{l_1}}w(x)(v_{n_l}-v_{n_m})\{\mathcal{F}(\nabla v_{n_l})^{p-1}\nabla_{\eta}\mathcal{F}(\nabla v_{n_l})-\mathcal{F}(\nabla v_{n_m})^{p-1}\nabla_{\eta}\mathcal{F}(\nabla v_{n_m})\}\nabla\phi\,dx\to 0
$$
as $l,m\to\infty$. Using the above estimates on $I$ and $J$, applying the Finsler algebraic inequality from Lemma \ref{alg} in \eqref{estlim}, the sequence $\{v_{n_l}\}$ converges strongly to $v$ in $W^{1,p}(\omega,w)$. Now we pass to the limit in \eqref{eapprox} and prove that $v$ is our required solution. To this end, assume that $\phi\in C_c^{1}(\Om)$ such that $\mathrm{supp}\,\phi\subset\Om_{l_1}$ for some $l_1\geq 1$. Then by the strong convergence of $v_{n_l}$ to $v$ in $W^{1,p}_{\mathrm{loc}}(\Om,w)$, we have
\begin{equation}\label{eplim1}
\lim_{l\to\infty}\int_{\Om}w(x)\mathcal{F}(\nabla v_{n_l})^{p-1}\nabla_{\eta}\mathcal{F}(\nabla v_{n_l})\nabla\phi\,dx=\int_{\Om}w(x)\mathcal{F}(\nabla v)^{p-1}\nabla_{\eta}\mathcal{F}(\nabla v)\nabla\phi\,dx.
\end{equation}
Moreover, by Lemma \ref{eexisapprox}, we obtain
$$
\left|h_{n_l} e^\frac{1}{\big(v_{n_l}+\frac{1}{n_l}\big)}\phi\right|\leq c h\in {L^1(\Om)},
$$
for some constant $c>0$ independent of $l$. By the Lebesgue dominated convergence theorem, we have
\begin{equation}\label{eplim2}
\lim_{l\to\infty}\int_{\Om}h_n(x) e^\frac{1}{\big(v_{n_l}(x)+\frac{1}{n_l}\big)}\phi(x)\,dx=\int_{\Om}h(x) e^\frac{1}{v(x)}\phi(x)\,dx.
\end{equation}
Thus, from \eqref{eplim1} and \eqref{eplim2}, we conclude that $v\in W^{1,p}_{loc}(\Om,w)$ is a weak solution of the problem $\mathcal{(R)}$. Further, it can be easily seen that $\{(v_{n_l}-\epsilon)^+\}$ is uniformly bounded in $W_0^{1,p}(\Om,w)$ for every $\epsilon>0$ and hence $(v-\epsilon)^+\in W_0^{1,p}(\Om,w)$. Since by Lemma \ref{eexisapprox}, we have $v_{n_l}\geq c_{\omega}>0$ for every $\omega\Subset\Om$, we obtain $v\geq c_{\omega}>0$ in $\omega$. Hence $v>0$ in $\Om$ and the result follows. \qed

\textbf{Proof of Theorem \ref{ethmnew}:} Noting Lemma \ref{AI} along with Remark \ref{ermk} and then following the same proof of Theorem \ref{ethm}, the result follows. \qed


\begin{thebibliography}{10}

\bibitem{ex1}
H.~Aimar, M.~Carena, R.~Dur\'{a}n, and M.~Toschi.
\newblock Powers of distances to lower dimensional sets as {M}uckenhoupt
  weights.
\newblock {\em Acta Math. Hungar.}, 143(1):119--137, 2014.

\bibitem{GFA}
Giovanni Anello, Francesca Faraci, and Antonio Iannizzotto.
\newblock On a problem of {H}uang concerning best constants in {S}obolev
  embeddings.
\newblock {\em Ann. Mat. Pura Appl. (4)}, 194(3):767--779, 2015.

\bibitem{arcoya}
David Arcoya and Lucio Boccardo.
\newblock Multiplicity of solutions for a {D}irichlet problem with a singular
  and a supercritical nonlinearities.
\newblock {\em Differential Integral Equations}, 26(1-2):119--128, 2013.

\bibitem{Merida}
David Arcoya and Lourdes Moreno-M\'{e}rida.
\newblock Multiplicity of solutions for a {D}irichlet problem with a strongly
  singular nonlinearity.
\newblock {\em Nonlinear Anal.}, 95:281--291, 2014.

\bibitem{Aubin}
Thierry Aubin.
\newblock Probl\`emes isop\'{e}rim\'{e}triques et espaces de {S}obolev.
\newblock {\em J. Differential Geometry}, 11(4):573--598, 1976.

\bibitem{BG}
Kaushik Bal and Prashanta Garain.
\newblock Multiplicity of solution for a quasilinear equation with singular
  nonlinearity.
\newblock {\em Mediterr. J. Math.}, 17(3):Paper No. 91, 20, 2020.

\bibitem{BGejde}
Kaushik Bal and Prashanta Garain.
\newblock Nonexistence results for weighted {$p$}-{L}aplace equations with
  singular nonlinearities.
\newblock {\em Electron. J. Differential Equations}, pages Paper No. 95, 12,
  2019.
  
\bibitem{BGaniso}
Kaushik Bal and Prashanta Garain.
\newblock Weighted and anisotropic sobolev inequality with extremal.
\newblock {\em Manuscripta Mathematica}, pages 1--17, 2021.

\bibitem{BGM}
Kaushik {Bal}, Prashanta {Garain}, and Tuhina {Mukherjee}.
\newblock {On an anisotropic p-Laplace equation with variable
singular exponent}.
\newblock {\em Advances in Differential Equations, Volume 26, Numbers 11--12 (2021), 535–-562}.

\bibitem{BCS}
D.~Bao, S.-S. Chern, and Z.~Shen.
\newblock {\em An introduction to {R}iemann-{F}insler geometry}, volume 200 of
  {\em Graduate Texts in Mathematics}.
\newblock Springer-Verlag, New York, 2000.

\bibitem{BFK}
M.~Belloni, V.~Ferone, and B.~Kawohl.
\newblock Isoperimetric inequalities, {W}ulff shape and related questions for
  strongly nonlinear elliptic operators.
\newblock volume~54, pages 771--783. 2003.
\newblock Special issue dedicated to Lawrence E. Payne.

\bibitem{BKmm}
M.~Belloni and B.~Kawohl.
\newblock A direct uniqueness proof for equations involving the {$p$}-{L}aplace
  operator.
\newblock {\em Manuscripta Math.}, 109(2):229--231, 2002.

\bibitem{BMM20}
Tesfa Biset, Benyam Mebrate, and Ahmed Mohammed.
\newblock A boundary-value problem for normalized {F}insler
  infinity-{L}aplacian equations with singular nonhomogeneous terms.
\newblock {\em Nonlinear Anal.}, 190:111588, 20, 2020.

\bibitem{BocOr}
Lucio Boccardo and Luigi Orsina.
\newblock Semilinear elliptic equations with singular nonlinearities.
\newblock {\em Calc. Var. Partial Differential Equations}, 37(3-4):363--380,
  2010.

\bibitem{Xiathesis}
Xia C.
\newblock {\em On a class of anisotropic problems, Dissertation zur Erlan-gung
  des Doktorgrades der FakultÃ¤t Mathematik undPhysik der
  Albert-Ludwigs-UniversitÃ¤tFreiburgimBreisgau,2012.https://freidok.uni-freiburg.de/fedora/objects/freidok:8693/datastreams/FILE1/content}.




\bibitem{Canino}
Annamaria Canino, Berardino Sciunzi, and Alessandro Trombetta.
\newblock Existence and uniqueness for {$p$}-{L}aplace equations involving
  singular nonlinearities.
\newblock {\em NoDEA Nonlinear Differential Equations Appl.}, 23(2):Art. 8, 18,
  2016.

\bibitem{CFR}
Giulio Ciraolo, Alessio Figalli, and Alberto Roncoroni.
\newblock Symmetry results for critical anisotropic {$p$}-{L}aplacian equations
  in convex cones.
\newblock {\em Geom. Funct. Anal.}, 30(3):770--803, 2020.

\bibitem{CRT}
M.~G. Crandall, P.~H. Rabinowitz, and L.~Tartar.
\newblock On a {D}irichlet problem with a singular nonlinearity.
\newblock {\em Comm. Partial Differential Equations}, 2(2):193--222, 1977.



\bibitem{DeCave}
Linda~Maria De~Cave.
\newblock Nonlinear elliptic equations with singular nonlinearities.
\newblock {\em Asymptot. Anal.}, 84(3-4):181--195, 2013.

\bibitem{a2}
Giuseppina {di Blasio}, Giovanni {Pisante}, and Georgeos {Psaradakis}.
\newblock {A weighted anisotropic Sobolev type inequality and its applications
  to Hardy inequalities}.
\newblock {\em arXiv e-prints}, page arXiv:1902.02091, February 2019.

\bibitem{DPV21}
Serena {Dipierro}, Giorgio {Poggesi}, and Enrico {Valdinoci}.
\newblock {Radial symmetry of solutions to anisotropic and weighted diffusion
  equations with discontinuous nonlinearities}.
\newblock {\em arXiv e-prints}, page arXiv:2105.02424, May 2021.

\bibitem{Santos}
Gelson C.~G. dos Santos, Giovany~M. Figueiredo, and Leandro~S. Tavares.
\newblock Existence results for some anisotropic singular problems via
  sub-supersolutions.
\newblock {\em Milan J. Math.}, 87(2):249--272, 2019.

\bibitem{Drabek}
Pavel Dr\'{a}bek, Alois Kufner, and Francesco Nicolosi.
\newblock {\em Quasilinear elliptic equations with degenerations and
  singularities}, volume~5 of {\em De Gruyter Series in Nonlinear Analysis and
  Applications}.
\newblock Walter de Gruyter \& Co., Berlin, 1997.

\bibitem{ex3}
Ricardo~G. Dur\'{a}n and Fernando L\'{o}pez~Garc\'{\i}a.
\newblock Solutions of the divergence and analysis of the {S}tokes equations in
  planar {H}\"{o}lder-{$\alpha$} domains.
\newblock {\em Math. Models Methods Appl. Sci.}, 20(1):95--120, 2010.

\bibitem{ex2}
Ricardo~G. Dur\'{a}n, Marcela Sanmartino, and Marisa Toschi.
\newblock Weighted a priori estimates for the {P}oisson equation.
\newblock {\em Indiana Univ. Math. J.}, 57(7):3463--3478, 2008.


\bibitem{a3}
A.~El~Hamidi and J.~M. Rakotoson.
\newblock Extremal functions for the anisotropic {S}obolev inequalities.
\newblock {\em Ann. Inst. H. Poincar\'{e} Anal. Non Lin\'{e}aire},
  24(5):741--756, 2007.

\bibitem{EP1}
G.~Ercole and G.~A. Pereira.
\newblock Fractional {S}obolev inequalities associated with singular problems.
\newblock {\em Math. Nachr.}, 291(11-12):1666--1685, 2018.

\bibitem{EP}
Grey Ercole and Gilberto de~Assis Pereira.
\newblock On a singular minimizing problem.
\newblock {\em J. Anal. Math.}, 135(2):575--598, 2018.

\bibitem{EFabes}
Eugene~B. Fabes, Carlos~E. Kenig, and Raul~P. Serapioni.
\newblock The local regularity of solutions of degenerate elliptic equations.
\newblock {\em Comm. Partial Differential Equations}, 7(1):77--116, 1982.


\bibitem{PF21}
Csaba {Farkas}, Alessio {Fiscella}, and Patrick {Winkert}.
\newblock {Singular Finsler double phase problems with nonlinear boundary
  condition}.
\newblock {\em arXiv e-prints}, page arXiv:2102.05467, February 2021.

\bibitem{PF20}
Csaba {Farkas} and Patrick {Winkert}.
\newblock {An existence result for singular Finsler double phase problems}.
\newblock {\em arXiv e-prints}, page arXiv:2011.03774, November 2020.

\bibitem{a1}
Stathis Filippas, Luisa Moschini, and Achilles Tertikas.
\newblock On a class of weighted anisotropic {S}obolev inequalities.
\newblock {\em J. Funct. Anal.}, 255(1):90--119, 2008.

\bibitem{Franzina}
Giovanni Franzina and Pier~Domenico Lamberti.
\newblock Existence and uniqueness for a {$p$}-{L}aplacian nonlinear eigenvalue
  problem.
\newblock {\em Electron. J. Differential Equations}, pages No. 26, 10, 2010.

\bibitem{G}
Prashanta {Garain}.
\newblock {On a degenerate singular elliptic problem}.
\newblock {\em (To appear in Mathematische Nachrichten)}, page
  arXiv:1803.02102, March 2018.

\bibitem{Garainaniso}
Prashanta Garain.
\newblock Existence and nonexistence results for anisotropic p-laplace equation
  with singular nonlinearities.
\newblock {\em Complex Variables and Elliptic Equations},https://doi.org/10.1080/17476933.2020.1801655


\bibitem{Garainpadm}
Prashanta {Garain}.
\newblock {Weighted singular problem with variable singular exponent and
  $p$-admissible weights}.
\newblock {\em arXiv e-prints}, page arXiv:2110.12049, October 2021.

\bibitem{GKin}
Prashanta Garain and Juha Kinnunen.
\newblock Nonexistence of variational minimizers related to a quasilinear
  singular problem in metric measure spaces.
\newblock {\em Proc. Amer. Math. Soc.}, 149(8):3407--3416, 2021.

\bibitem{GM}
Prashanta Garain and Tuhina Mukherjee.
\newblock On a class of weighted {$p$}-{L}aplace equation with singular
  nonlinearity.
\newblock {\em Mediterr. J. Math.}, 17(4):Paper No. 110, 18, 2020.

\bibitem{GU}
Prashanta Garain and Alexander Ukhlov.
\newblock Mixed local and nonlocal Sobolev inequalities with extremal and associated quasilinear singular elliptic problems.
\newblock Preprint, 2021.

\bibitem{GRbook}
Marius Ghergu and Vicen\c{t}iu~D. R\u{a}dulescu.
\newblock {\em Singular elliptic problems: bifurcation and asymptotic
  analysis}, volume~37 of {\em Oxford Lecture Series in Mathematics and its
  Applications}.
\newblock The Clarendon Press, Oxford University Press, Oxford, 2008.

\bibitem{GST}
Jacques Giacomoni, Ian Schindler, and Peter Tak\'{a}\v{c}.
\newblock Sobolev versus {H}\"{o}lder local minimizers and existence of
  multiple solutions for a singular quasilinear equation.
\newblock {\em Ann. Sc. Norm. Super. Pisa Cl. Sci. (5)}, 6(1):117--158, 2007.




\bibitem{Hara}
Takanobu {Hara}.
\newblock {Trace inequalities of the Sobolev type and nonlinear Dirichlet
  problems}.
\newblock {\em arXiv e-prints}, page arXiv:2102.09697, February 2021.


\bibitem{Juh}
Juha Heinonen, Tero Kilpel\"{a}inen, and Olli Martio.
\newblock {\em Nonlinear potential theory of degenerate elliptic equations}.
\newblock Oxford Mathematical Monographs. The Clarendon Press, Oxford
  University Press, New York, 1993.
\newblock Oxford Science Publications.


\bibitem{Tero}
Tero Kilpel\"ainen.
\newblock Weighted {S}obolev spaces and capacity.
\newblock {\em Ann. Acad. Sci. Fenn. Ser. A I Math.}, 19(1):95--113, 1994.

\bibitem{Stam}
David Kinderlehrer and Guido Stampacchia.
\newblock {\em An introduction to variational inequalities and their
  applications}, volume~88 of {\em Pure and Applied Mathematics}.
\newblock Academic Press, Inc. [Harcourt Brace Jovanovich, Publishers], New
  York-London, 1980.

\bibitem{Mirifixed}
Ahmed~R\'{e}da Leggat and Sofiane El-Hadi Miri.
\newblock Anisotropic problem with singular nonlinearity.
\newblock {\em Complex Var. Elliptic Equ.}, 61(4):496--509, 2016.

\bibitem{LzMc}
A.~C. Lazer and P.~J. McKenna.
\newblock On a singular nonlinear elliptic boundary-value problem.
\newblock {\em Proc. Amer. Math. Soc.}, 111(3):721--730, 1991.


\bibitem{PLin}
Peter Lindqvist.
\newblock Addendum: ``{O}n the equation {${\rm div}(|\nabla u|^{p-2}\nabla
  u)+\lambda|u|^{p-2}u=0$}'' [{P}roc. {A}mer. {M}ath. {S}oc. {\bf 109} (1990),
  no. 1, 157--164; {MR}1007505 (90h:35088)].
\newblock {\em Proc. Amer. Math. Soc.}, 116(2):583--584, 1992.

\bibitem{MV}
I.-I. Mezei and O.~Vas.
\newblock Existence results for some {D}irichlet problems involving
  {F}insler-{L}aplacian operator.
\newblock {\em Acta Math. Hungar.}, 157(1):39--53, 2019.

\bibitem{Mikko}
Pasi Mikkonen.
\newblock On the {W}olff potential and quasilinear elliptic equations involving
  measures.
\newblock {\em Ann. Acad. Sci. Fenn. Math. Diss.}, (104):71, 1996.

\bibitem{Mirivar}
Sofiane El-Hadi Miri.
\newblock On an anisotropic problem with singular nonlinearity having variable
  exponent.
\newblock {\em Ric. Mat.}, 66(2):415--424, 2017.




\bibitem{Muc}
Benjamin Muckenhoupt.
\newblock Weighted norm inequalities for the {H}ardy maximal function.
\newblock {\em Trans. Amer. Math. Soc.}, 165:207--226, 1972.

\bibitem{Oljama}
Francescantonio Oliva.
\newblock Regularizing effect of absorption terms in singular problems.
\newblock {\em J. Math. Anal. Appl.}, 472(1):1136--1166, 2019.

\bibitem{Olsaim}
Francescantonio Oliva and Francesco Petitta.
\newblock On singular elliptic equations with measure sources.
\newblock {\em ESAIM Control Optim. Calc. Var.}, 22(1):289--308, 2016.

\bibitem{Oljde}
Francescantonio Oliva and Francesco Petitta.
\newblock Finite and infinite energy solutions of singular elliptic problems:
  existence and uniqueness.
\newblock {\em J. Differential Equations}, 264(1):311--340, 2018.

\bibitem{OrPet}
Luigi Orsina and Francesco Petitta.
\newblock A {L}azer-{M}c{K}enna type problem with measures.
\newblock {\em Differential Integral Equations}, 29(1-2):19--36, 2016.

\bibitem{Otani}
Mitsuharu \^{O}tani.
\newblock Existence and nonexistence of nontrivial solutions of some nonlinear
  degenerate elliptic equations.
\newblock {\em J. Funct. Anal.}, 76(1):140--159, 1988.

\bibitem{Pe}
Ireneo Peral.
\newblock Multiplicity of solutions for the p-laplacian, lecture notes at the
  second school on nonlinear functional analysis and applications to
  differential equations at ictp of trieste.
\newblock {\em ICTP lecture notes}, 1997.

\bibitem{PeSil}
Kanishka Perera and Elves A.~B. Silva.
\newblock On singular {$p$}-{L}aplacian problems.
\newblock {\em Differential Integral Equations}, 20(1):105--120, 2007.


\bibitem{Rbook}
R.~Tyrrell Rockafellar.
\newblock {\em Convex analysis}.
\newblock Princeton Mathematical Series, No. 28. Princeton University Press,
  Princeton, N.J., 1970.

\bibitem{Talenti}
Giorgio Talenti.
\newblock Best constant in {S}obolev inequality.
\newblock {\em Ann. Mat. Pura Appl. (4)}, 110:353--372, 1976.

\end{thebibliography}
\end{document}